\DeclareSymbolFont{AMSb}{U}{msb}{m}{n} 
\documentclass[11pt]{amsart}

\usepackage[vmargin=2.5cm,hmargin=2.5cm]{geometry}

\newtheorem{theorem}{Theorem}[section]
\newtheorem{lemma}[theorem]{Lemma}
\newtheorem{corollary}[theorem]{Corollary}
\newtheorem{proposition}[theorem]{Proposition}

\theoremstyle{remark}
\newtheorem{remark}[theorem]{Remark}
\newtheorem{example}[theorem]{Example}

\DeclareMathOperator{\Ap}{Ap}

\usepackage[utf8]{inputenc}    
\usepackage[T1]{fontenc}       

\usepackage{fourier}

\usepackage{enumitem}
\setlist[enumerate]{leftmargin=*, listparindent=\parindent, parsep=0pt, font=\upshape, label=\alph*)} 
\setlist[itemize]{leftmargin=*} 

\usepackage[backend=biber]{biblatex}
\renewbibmacro{in:}{}

\usepackage{comment}
\usepackage{float}
\usepackage{url}
\usepackage{xcolor}

\newcommand\andres[1]{{\textcolor{magenta}{#1}}}

\title[Isolated factorizations and their applications]{Isolated factorizations and their applications in simplicial affine semigroups} 

\author[Garc\'ia-S\'anchez]{Pedro A. Garc\'ia-S\'anchez}
\address{Departamento de \'Algebra, Universidad de Granada, E-18071 Granada, Espa\~na}
\email{pedro@ugr.es}
\thanks{The first author is supported by the project MTM2014-55367-P, which is funded by Ministerio de Econom\'{\i}a y Competitividad and Fondo Europeo de Desarrollo Regional FEDER, and by the Junta de Andaluc\'{\i}a Grant Number FQM-343}
\author[Herrera-Poyatos]{Andr\'es Herrera-Poyatos}
\address{Departamento de \'Algebra, Universidad de Granada, E-18071 Granada, Espa\~na}
\email{andreshp9@gmail.com}
\thanks{The second author is supported by the fellowship ``Beca de Iniciaci\'on a la Investigaci\'on para Estudiantes de Grado del Plan Propio 2017'' (Vicerrectorado de Investigación y Transferencia, Universidad de Granada).}

\begin{document}
\date{}
\maketitle
\begin{abstract}
We introduce the concept of isolated 
factorizations of an element of a commutative monoid and study its properties. We give several bounds for the number of isolated factorizations of simplicial affine semigroups and numerical semigroups. We also generalize $\alpha$-rectangular numerical semigroups to the context of simplicial affine semigroups and study their isolated factorizations. As a consequence of our results, we characterize those complete intersection simplicial affine semigroups with only one Betti minimal element in several ways. Moreover, we define Betti sorted and Betti divisible simplicial affine semigroups and characterize them in terms of gluings and their minimal presentations. Finally, we determine all the Betti divisible numerical semigroups, which turn out to be those numerical semigroups that are free for any arrangement of their minimal generators.



\end{abstract}


\section{Introduction} \label{sec:intro}

Let $(M,+)$ be a commutative monoid ($+$ is a binary operation that is associative and commutative, and has an idendity element, denoted by $0$). We say that $M$ is \emph{reduced} or \emph{unit free} if whenever $a+b=0$, with $a,b\in M$, then $a=0=b$, that is, the only unit in $M$ is the identity element. The monoid $M$ is \emph{cancellative} if $a+b=a+c$ implies $b=c$ for every $a,b,c\in M$. Every monoid in this manuscript is commutative, reduced and cancellative; thus in the sequel we will omit these adjectives. 

An element $a\in M$ is an \emph{atom} of $M$ if $a=b+c$ for some $b,c\in M$ implies $0\in \{b,c\}$ (recall that we are assuming that $M$ is reduced, otherwise the definition is slightly different). We will denote by $\mathcal{A}(M)$ the set of atoms of $M$. The monoid $M$ is \emph{atomic} if $M=\langle \mathcal{A}(M)\rangle = \{ \sum_{i=1}^n a_i \colon n\in \mathbb N, a_1,\ldots, a_n\in \mathcal{A}(M)\}$, that is, every element in $M$ can be expressed as a finite sum of atoms.

We define the following relation on $M$. Given $a,b\in M$, we write $a\le_M b$ if there exists $c\in M$ such that $a+c=b$. Since we are assuming $M$ to be cancellative, this relation is an order relation. Also cancellativity allows us to express $a\le_M b$ with $b-a\in M$, as in this setting $M$ is naturally embedded in its quotient group. 

The monoid $M$ satisfies the \emph{ascending chain condition on principal ideals} if every chain of the form $m_1+M\subseteq m_2+M\subseteq \cdots \subseteq m_k+M\subseteq \dots$ becomes stationary. This is equivalent to saying that there is no infinite strictly descending chain with respect to $\le_M$. From \cite[Proposition 1.1.4]{geroldinger-hk} a monoid with the ascending chain condition on principal ideals is atomic. 

If $M$ is atomic, and we denote by $A$ its set of atoms, then there is a natural epimorphism $\varphi:\mathbb{N}^{(A)}\to S$, defined as $\varphi((n_a)_{a\in A})=\sum_{a\in A} n_a a$; here $\mathbb{N}^{(A)}$ means the direct product of $\sharp A$ copies of $\mathbb{N}$, or the set of sequences of nonnegative integers indexed in $A$ with all its entries equal to zero except for finitely many of them. Therefore, $M$ is isomorphic as a monoid to $\mathbb{N}^{(A)}/\ker\varphi$, where $\ker\varphi=\{(a,b)\in \mathbb{N}^{(A)}\times \mathbb{N}^{(A)} \colon \varphi(a)=\varphi(b)\}$. A \emph{presentation} for $M$ is a system of generators of $\ker\varphi$ as a congruence. A \emph{minimal presentation} is a presentation such that any of its proper subsets does not generate $\ker\varphi$. If $M$ has the ascending chain condition on principal ideals, then all minimal presentations have the same cardinality \cite{min-pres-mac}.

For $m\in M$, let $\mathrm{Z}(m)$ be the \emph{set of factorizations} of $m$ in $M$, that is, the fiber $\varphi^{-1}(m)$. For a subset $T$ of $M$, set $\mathrm{Z}(T)=\bigcup_{m\in T}\mathrm{Z}(m)$.
Let $\nabla_m$ be the graph with vertices $\mathrm{Z}(m)$, and whose edges are the pairs $(x,y)$ such that $x\cdot y\neq 0$ (that is, edges join factorizations having minimal generators in common; $x\cdot y$ denotes the dot product of $x$ and $y$). The connected components of $\nabla_m$ are called the \emph{R-classes of $m$}. The \emph{Betti elements} of $M$ are those elements with at least two R-classes. The set of the Betti elements of $M$ is denoted by $\mathrm{Betti}(M)$. Betti elements and their R-classes can be used to characterize the minimal presentations of $M$, see Section \ref{sec:pre:presentations}, and, thus, they are widely studied in the theory of commutative monoids (see for instance \cite{monoids:up, overview-nuf}). A Betti element $b \in M$ is \emph{minimal} if $b \in \mathrm{Minimals}_{\le_M} \mathrm{Betti}(M)$.


We say that a factorization $x \in \mathrm{Z}(m)$, $m\in M$, is \emph{isolated} if it is disjoint with every other factorization of $m$, that is, the R-class of $x$ in $\nabla_m$ is singleton. It easily follows that if $m$ is an element of $M$ having an isolated factorization, then either $m$ has a unique expression (or admits a unique factorization), that is, $\mathsf{Z}(m)$ is a singleton, or $m$ is a Betti element. In this paper we investigate the properties of these factorizations. 
As a consequence of this study, we are able to bound the number of isolated factorizations of simplicial affine semigroups and numerical semigroups in several ways (see Section \ref{sec:pre} for definitions). 

Isolated factorizations have several applications. These applications come from the connection between the minimal presentations and the Betti elements of a monoid and are particularly interesting in the case of complete intersection affine semigroups, whose definition is recalled in Section \ref{sec:pre:affine}. Complete intersection semigroups are relevant outside the theory of numerical semigroups \cite{ns-app, telescopic, fr:telescopic, fr:two} and, consequently, they have been the main topic of several research papers.
In the search of families of complete intersection semigroups, Bertin and Carbonne introduced in \cite{free} the concept of free numerical semigroups, which allows to construct complete intersection numerical semigroups of any desired embedding dimension. 
Since then many families of free numerical semigroups have been studied; see, for instance, \cite{ci:classes:1, ci:classes:2} and the references given therein. One of these families is that of $\alpha$-rectangular numerical semigroups \cite{ci:classes:2}. In this paper we give a wider definition of $\alpha$-rectangularity and generalize it to the context of simplicial affine semigroups. Moreover, we show that complete intersection numerical semigroups with only one Betti minimal element are $\alpha$-rectangular under our definition (Corollary \ref{cor:ci-b1}). A more general result for simplicial affine semigroups is also given (Theorem \ref{thm:alpha-m+1}).

We also introduce two new families of commutative monoids: Betti sorted and Betti divisible monoids. 
A commutative monoid is \emph{Betti sorted} (respectively \emph{Betti divisible}) if its Betti elements are totally ordered with respect to the relation $\le_M$  (respectively with respect to divisibility).  Isolated and Betti restricted factorizations allow us to deeply investigate these two families in the case of simplicial affine semigroups. More concretely, in Theorem \ref{thm:betti-sorted:alpha} we show that these semigroups are $\alpha$-rectangular under some extra hypothesis.
Besides, we characterize these semigroups in terms of their minimal presentations (Theorem \ref{thm:betti-sorted}) and gluings (Corollary \ref{cor:betti-sorted:gluing}). 
We also provide similar results for Betti divisible numerical semigroups. It is worth mentioning that simplicial affine semigroups with a single Betti element are Betti divisible. These semigroups were studied in \cite{single-betti}.

When it comes to numerical semigroups, our results can be easily interpreted. In particular, the following strict inclusions among the families of numerical semigroups studied hold without any extra assumptions:
\begin{gather*}
\text{Unique Betti element} \subset \text{Betti divisible} \subset \text{Betti sorted} \\
\subset \text{Complete intersection with only one Betti minimal element} \\
\subset \alpha\text{-rectangular} \subset \text{Free} \subset \text{Complete intersection}.    
\end{gather*}
 Moreover, we characterize Betti divisible numerical semigroups in terms of their minimal system of generators (Theorem \ref{thm:betti-divisible:generators}). This theorem generalizes the main result of \cite{single-betti}, which states that numerical semigroups with only one Betti element are those generated by products of the form $ \prod_{j \ne i} a_j$ with $a_1,\ldots,a_e$ coprime integers. We also prove that Betti divisible numerical semigroups are free for any arrangement of their minimal generators (Theorem \ref{thm:betti-divisible:free}). As a consequence of all the characterizations presented, we show that numerical semigroups with a single Betti element are those that are $\alpha$-rectangular for any of their minimal generators (Theorem \ref{thm:single-betti-alpha}).

The paper is organized as follows. In Section \ref{sec:pre} we provide some notation and recall some definitions and properties that are used in our work. In Section \ref{sec:isolated} we study the properties of isolated factorizations. In Section \ref{sec:alpha} we generalize $\alpha$-rectangular numerical semigroups to the context of simplicial affine semigroups. In Section \ref{sec:ci} we study complete intersection simplicial affine semigroups with only one Betti minimal element. 
Finally, in Sections \ref{sec:betti-sorted} and \ref{sec:betti-divisible} we delve into Betti sorted and Betti divisible numerical semigroups respectively.


Computations are performed with GAP \cite{gap} and, more concretely, the package \texttt{NumericalSgps} \cite{numericalsgps}. The functions and code used in this paper will be included in the next release of \texttt{NumericalSgps}. 

\section{Preliminaries} \label{sec:pre}

\subsection{Numerical semigroups}

A numerical semigroup $S$ is an additive submonoid of the nonnegative integers $\mathbb{N} = \{0,1,2, \ldots\}$ such that its complement in $\mathbb{N}$ is finite. The maximum integer that is not in $S$ is the \emph{Frobenius number}, denoted by $\mathrm{F}(S)$. Numerical semigroups have an unique finite minimal system of generators, which we usually write as $\{n_1,\dots, n_e\}$; this set is precisely $\mathcal{A}(S)$. The integers $e$ and $\min(S\setminus\{0\})$ are the \emph{embedding dimension} and the \emph{multiplicity} of $S$. They are referred to as $\mathrm{e}(S)$ and $\mathrm{m}(S)$ respectively. For an introduction to numerical semigroups see \cite[Chapter 1]{ns}.

\subsection{Affine semigroups} \label{sec:pre:affine}

An affine semigroup $S$ is a finitely generated submonoid of $\mathbb{N}^r$ for some positive integer $r$. Let $S$ be an affine semigroup of $\mathbb{N}^r$. Let $\mathcal{G}(S)$ denote the subgroup of $\mathbb{Z}^r$ generated by $S$. The \emph{dimension} of $S$ is the rank of the group $\mathcal{G}(S)$ and it is denoted by $\mathrm{dim}(S)$. As with numerical semigroups, the embedding dimension of an affine semigroup is the cardinality of its minimal generating set. The \emph{codimension} of $S$ is the integer $\mathrm{e}(S) - \mathrm{dim}(S)$. The semigroup $S$ is a \emph{simplicial} semigroup of $\mathbb{N}^r$ if there is an arrangement of its minimal generators, $\{n_1, \ldots, n_r, n_{r+1}, \ldots, n_{r+m}\}$, such that $\mathrm{L}_{\mathbb{Q}^+}(S) = \mathrm{L}_{\mathbb{Q}^+}(\{n_1, \ldots, n_r\})$,
where 
\[\mathrm{L}_{\mathbb{Q}^+}(A) = \left\{\sum\nolimits_{a \in B} \lambda_a a : B \subseteq A \text{ is finite}, \lambda_a \in \mathbb{Q}^+ \text{ for all } a \in B\right\}.\]
The natural number $m$ is the codimension of $S$. From now on when we say that $\{n_1, \ldots, n_{r+m}\}$ is a minimal system of generators of $S$ we assume that $r = \mathrm{dim}(S)$ and $\mathrm{L}_{\mathbb{Q}^+}(S) = \mathrm{L}_{\mathbb{Q}^+}(\{n_1, \ldots, n_r\})$.
Note that a numerical semigroup $S$ is an affine semigroup for any order of its minimal generators. Its dimension is $1$ and its codimension is $\mathrm{e}(S)-1$.

Since affine semigroups are finitely generated commutative monoids, they are finitely presented \cite{redei} (see also \cite{monoids} for a proof following the notation adopted in this paper). It is not difficult to show that affine semigroups have the ascending chain condition on principal ideals and, thus, all the minimal presentations of an affine semigroup have the same cardinality, which is always greater or equal than its codimension. Affine semigroups which attain this bound are called \emph{complete intersections} \cite{ci}. 


\subsection{Ap\'ery sets}\label{sec:pre:apery}
Let $M$ be a commutative monoid. For any $m\in M$ one can define the Ap\'ery set of $m$ in $M$ as 
\[\Ap(M; m) = M\setminus (m+M).
\] 
Ap\'ery sets are a powerful tool to describe monoids, and they will often appear in our work. If $M$ is reduced (the only unit is zero), cancellative and has the ascending chain condition on principal ideals, then we get that every $n\in M$ can be expressed uniquely as $n=k m +\omega$ for some nonnegative integer $k$ and $\omega\in \Ap(M;m)$ \cite{min-pres-mac}. For a subset $N$ of $M$, we define 
\[
\Ap(M;N)=\bigcap_{n\in N}\Ap(M;n).
\]

Let $S$ be a numerical semigroup and $n \in S$. Note that $0 \in \Ap(S; n)$ and $\max \Ap(S; n)-n = \mathrm{F}(S)$, the Frobenius number of $S$. In addition, the Ap\'ery set of $n$ in $S$ has cardinality equal to $n$, \cite[Chapter 1]{ns}. For other type of monoids Ap\'ery sets may not be finite. This is the case of simplicial affine semigroups. Nonetheless, if $S$ is a simplicial affine semigroup of $\mathbb{N}^r$ minimally generated by $\{n_1, \ldots, n_{r+m}\}$, then we have 
\begin{equation*} \label{eq:c*}
    \Ap(S; \{n_1, \ldots, n_r\}) \subseteq \{\lambda_{r+1} n_{r+1} + \cdots \lambda_{r+m} n_{r+m} : 0 \le \lambda_{r+1} < c_{r+1}^*, \ldots, 0 \le \lambda_{r+m} < c_{r+m}^*\},
\end{equation*}
where $c_i^* = \min \{k \in \mathbb{Z}^+ : k n_i \in \langle n_1, \ldots, n_{r+i-1} \rangle \}$ for every $i \in \{r+1,\ldots,r+m\}$ (see \cite[Lemma 1.4]{free-affine}). The fact that $\mathrm{L}_{\mathbb{Q}^+}(S) = \mathrm{L}_{\mathbb{Q}^+}(\{n_1, \ldots, n_r\})$ ensures the existence of $c_{r+1}^*, \ldots, c_{r+m}^*$. Therefore, the set $\Ap(S; \{n_1, \ldots, n_r\})$ is finite.

\subsection{Simplicial Cohen-Macaulay affine semigroups}

Simplicial Cohen-Macaulay affine semigroups are characterized in several ways in \cite{cm}. One of these characterizations involves Ap\'ery sets and it is the one that we use in our work. We include this characterization in the following result, where we give another handy equivalence that will be used later.

\begin{proposition} \label{prop:cm}
Let $S$ be a simplicial affine semigroup of $\mathbb{N}^r$ minimally generated by $\{n_1, \ldots, n_{r+m}\}$. The following statements are equivalent:
  \begin{enumerate}
      \item \label{item:cm:cm} $S$ is Cohen-Macaulay;
      \item \label{item:cm:unique:1} for every $s \in S$ there are unique non negative integers $\lambda_1, \ldots, \lambda_r$ and $\omega \in Ap(S; \{n_1, \ldots, n_r\})$ such that $s = \omega + \lambda_1 n_1 + \cdots + \lambda_r n_r$;
      \item \label{item:cm:unique:2} for every $s \in S$ there are unique $n \in \langle n_1, \ldots, n_r \rangle$ and $\omega \in Ap(S; \{n_1, \ldots, n_r\})$ such that $s = \omega + n$.
  \end{enumerate}
\end{proposition}
\begin{proof}
  See \cite[Theorem 1.5]{cm} for a proof of the equivalence between \ref{item:cm:cm} and \ref{item:cm:unique:1}. The equivalence between \ref{item:cm:unique:1} and \ref{item:cm:unique:2} follows easily from the fact that $\{n_1,\ldots, n_r\}$ are linearly independent.
\end{proof}

In particular, since \ref{item:cm:unique:1} is verified for numerical semigroups, they are Cohen-Macaulay. The uniqueness stated in \ref{item:cm:unique:1} is the reason why working with simplicial Cohen-Macaulay affine semigroups is in many cases similar to working with numerical semigroups (compare with \cite[Lemma 2.6]{ns}).

A Cohen-Macaulay simplicial affine semigroup $S$ of $\mathbb{N}^r$ minimally generated by $\{ n_1,\ldots,n_{r+m}\}$ is said to be Gorenstein if $\Ap(S;\{n_1,\ldots,n_r\})$ has a unique maximal element (in this case the semigroup ring of $S$ is Gorenstein; thus the name, see \cite[Section 4]{cm}).
 
\subsection{Betti elements and minimal presentations} \label{sec:pre:presentations}

In Section \ref{sec:intro} we recalled the definitions of Betti elements and minimal presentations of a monoid. In this section we describe a well-known characterization of all the minimal presentations in terms of the Betti elements and their R-classes. 
This characterization was introduced by Eliahou in \cite{eliahou-th}, and it is equivalent to the one given by Rosales in \cite{rosales-min-pres}. This procedure was later extended to strongly reduced monoids in \cite{monoids:presen}, and later to the more general setting of monoids with the minimal ascending chain condition on principal ideals, \cite{min-pres-mac}. 

Let $M$ be a monoid with the ascending chain condition on principal ideals, and let us denote by $A$ its set of atoms. 
We denote the set of Betti elements of $M$ by $\mathrm{Betti}(M)$. For $m \in M$, let $\mathrm{nc}(\nabla_m)$ be the number of connected components of $\nabla_m$, and let $\mathfrak{d}(m)$ be the cardinality of $\mathrm{Z}(m)$, which is also known as the \emph{denumerant} of $m$ in $M$. If $M$ is a numerical or an affine semigroup, then these amounts are finite. 

Given $\rho\subseteq \mathbb{N}^{(A)}\times \mathbb{N}^{(A)}$, we use $\mathrm{cong}(\rho)$ to denote the congruence generated by $\rho$, that is, the intersection of all congruences containing $\rho$. Define \[\rho^0 = \rho \cup \{(y,x) : (x,y) \in \rho\},\ \rho^1 = \{(x+u, y+u) : (x,y) \in \rho^0, u \in \mathbb{N}^{(A)}\}.\] It turns out that $\mathrm{cong}(\rho)$ is the transitive closure of $\rho^1$.

A minimal presentation of $M$ can be constructed as follows.  Assume that $m\in\mathrm{Betti}(M)$ and set $r=\mathrm{nc}(\nabla_m)$. 
Let $(X_i, Y_i)$, $i\in I$ be the edges of a tree whose vertices are the connected components of $\nabla_m$. Pick $x_i\in X_i$ and $y_i \in Y_i$ for all $i\in I$ and set $\rho^{(m)}=\{(x_i,y_i) \colon i\in I\}$. Then, $\rho=\bigcup_{m\in\mathrm{Betti}(M)} \rho^{(m)}$ is a minimal presentation of $M$ (see for instance \cite[Chapter 4]{ns-app} for numerical semigroups or \cite{min-pres-mac} for the general case). All minimal presentations have this form. As a consequence, all minimal presentations have cardinality equal to $\sum_{s\in\mathrm{Betti}(M)}(\mathrm{nc}(\nabla_m)-1)$. 


\subsection{Gluings} 

Let $S_1$ and $S_2$ be two numerical semigroups, and $a_1$, $a_2$ be two coprime integers such that $a_1\in S_2$, $a_2\in S_1$ and neither $a_1$ nor $a_2$ is a minimal generator. The numerical semigroup $a_1S_1+a_2S_2$ is a \emph{gluing} of $S_1$ and $S_2$. We will write $S=a_1S_1+_{a_1a_2}a_2S_2$. It turns out that the embedding dimension of $S$ is the sum of the embedding dimensions of $S_1$ and $S_2$. Moreover, we can obtain any minimal presentation of $S$ from the union of a minimal presentation of $a_1S_1$, a minimal presentation of $a_2 S_2$ and a singleton set containing a pair of factorizations of $a_1 a_2$, one in $a_1S_1$ and the other in $a_2S_2$ \cite{monoids:up}. As a consequence, we obtain
\begin{equation} \label{eq:betti}
\mathrm{Betti}(S)=\{a_1a_2\}\cup\{ a_1b_1\colon b_1\in \mathrm{Betti}(S_1)\}\cup \{a_2b_2\colon b_2\in \mathrm{Betti}(S_2)\}.
\end{equation}

Gluings were introduced by Rosales in \cite{gluing}. The definition of gluing was motivated by the famous characterization of complete intersection numerical semigroups given by Delorme \cite{ci:char}: a numerical semigroup is complete intersection if and only if it is either $\mathbb{N}$ or a gluing of two complete intersection numerical semigroups. In this paper we obtain similar results for the families of numerical semigroups considered, see Corollaries \ref{cor:alpha:gluing}, \ref{cor:ci-b1:gluing}, \ref{cor:betti-sorted:gluing} and \ref{cor:betti-divisible:gluing}.


The concept of gluing can be generalized to the more general setting of reduced cancellative commutative monoids with the ascending chain condition on principal ideals, see \cite{min-pres-mac}. Let $M$ be a monoid with the ascending chain condition on principal ideals. Denote by $A$ its set of atoms. Assume that $\{A_1,A_2\}$ is a nontrivial partition of $A$, and let $M_i=\langle A_i\rangle$, $i\in \{1,2\}$. We say that $M$ is the gluing of $M_1$ and $M_2$ by $d\in M$ if $M$ admits a presentation of the form $\sigma=\sigma_1\cup \sigma_2\cup \{(a,b)\}$ with $\sigma_i\subseteq \mathsf{Z}(M_i)$ a presentation for $M_i$, and $a\in \mathsf{Z}(d)\cap\mathsf{Z}(M_1)$ and $b\in \mathsf{Z}(d)\cap \mathsf{Z}(M_2)$. If this is the case, then every presentation of $M$ has this form. It follows (see \cite{min-pres-mac}) that 
\begin{equation}\label{eq:betti-gen}
    \mathrm{Betti}(M)=\mathrm{Betti}(M_1)\cup \mathrm{Betti}(M_2)\cup \{d\}.
\end{equation}
Observe that the above definition generalizes that of numerical semigroups ($M_1=a_1S_1$, $M_2=a_2S_2$ and $d=a_1a_2$).

\subsection{Free simplicial affine semigroups} \label{sec:pre:free}

Free numerical semigroups were introduced by Bertin and Carbonne in \cite{free}.  See for instance \cite[Chapter 8]{ns} for several characterizations of these semigroups, or \cite{free-affine} for a generalization of these semigroups in the setting of affine semigroups. 

Let $S$ be a simplicial affine semigroup minimally generated by $\{n_1, \ldots, n_{r+m}\}$. For every $i \in \{r+1, \ldots, r+m\}$ one can define the following constants:
\begin{itemize}
	\item $\bar{c}_i = \min \{c \in \mathbb{Z}^+ : c n_i \in \mathcal{G}(\{n_1, \ldots, n_{i-1}\})\}$,
	\item $c_i^* = \min \{c \in \mathbb{Z}^+ : c n_i \in \langle n_1, \ldots, n_{i-1} \rangle\}$,
	\item $c_i = \min \{c \in \mathbb{Z}^+ : c n_i \in \langle n_1, \ldots,n_{i-1},n_{i+1},\ldots n_{r+m} \rangle\}$.
\end{itemize}

Note that $\bar{c}_i \le c_i^*$ and $c_i \le c_i^*$ for every $i \in \{r+1, \ldots, r+m\}$. The semigroup $S$ is \emph{free} for the arrangement $(n_1, \ldots, n_{r+m})$ if $\bar{c}_i n_i \in \langle n_1, \ldots, n_{i-1} \rangle$ (equivalently, $\bar{c}_i=c_i^*$) for every $i \in \{r+1, \ldots, r+m\}$. 

Cohen-Macaulay free simplicial affine semigroups are characterized in several ways \cite{free-affine}, some of them are included in the following result. We use $\mathbf{e}_i$ to denote the $i^{th}$ row of the $(r+m) \times (r+m)$ identity matrix.


\begin{lemma} \label{lem:free}
  Let $S$ be a Cohen-Macaulay and simplicial affine semigroup of $\mathbb{N}^r$ minimally generated by $\{n_1, \ldots, n_{r+m}\}$. The following statements are equivalent:
  \begin{enumerate}
      \item $S$ is free (for this arrangement of the minimal generators);
      \item $\#\Ap(S; \{n_1, \ldots, n_r\}) = c_{r+1}^* \cdots c_{r+m}^*$;
      \item $\Ap(S; \{n_1, \ldots, n_r\}) = \{ \lambda_{r+1} n_{r+1} + \cdots + \lambda_{r+m} n_{r+m} : 0 \le \lambda_{i} < c_i^*, i \in \{r+1,\ldots, r+m\}\}$;
      \item $\max_{\le_S}(\Ap(S; \{n_1, \ldots, n_r\})) = (c_{r+1}^* - 1) n_{r+1} + \cdots + (c_{r+m}^*-1) n_{r+m}$;
      \item  $S' = \langle n_1, \ldots, n_{r+m-1} \rangle$ is free for $(n_1, \ldots, n_{r+m-1})$ and $S$ is a gluing of $S'$ and $\langle n_{r+m} \rangle$;
      \item $S$ admits a minimal presentation of the form 
      	\[\left\{(c^*_i \mathbf{e}_i, a_{i_1} \mathbf{e}_1 + \cdots + a_{i_{i-1}} \mathbf{e}_{i-1}) \colon i \in \{r+1, \ldots, r+m\}\right\},\] for some $a_{i_j}\in \mathbb{N}$;      \item $S$ admits a minimal presentation of the form 
      	\[\left\{(k_i \mathbf{e}_i, a_{i_1} \mathbf{e}_1 + \cdots + a_{i_{i-1}} \mathbf{e}_{i-1}) \colon i \in \{r+1, \ldots, r+m\}\right\},\] for some $k_i\in \mathbb{N}\setminus\{0\}$ and some $a_{i_j}\in \mathbb{N}$.
  \end{enumerate}
\end{lemma}

As a consequence, we obtain that these semigroups are complete intersections.

\begin{remark}
If $S$ is a free numerical semigroup for the arrangement $(n_1,\ldots,n_e)$ of its minimal generators, then it follows easily from the above result, the fact that $\#\Ap(S,n_1)=n_1$, and Selmer's formula for the Frobenius number that 
\begin{enumerate}
  	\item $n_1 = c_2^* \cdots c_e^*$;
  	\item $\mathrm{F}(S) + n_1 = (c_2^* -1) n_2 + \cdots + (c_e^* -1) n_e$.
\end{enumerate}
\end{remark}

\begin{remark}
If $S$ is a numerical semigroup minimally generated by $\{n_1, \ldots, n_e\}$, then we have $\bar{c}_i = \min \{c \in \mathbb{Z}^+ : c n_i \text{ is a multiple of } d_i\}$, where $d_i = \gcd \{n_1, \ldots, n_{i-1}\}$, for every $i \in \{2, \ldots, e\}$. Note that $\bar{c}_i = d_i / d_{i+1}$ and $n_1 = \bar{c}_2 \cdots \bar{c}_e$. It is not difficult to show that the integers $\bar{c}_2, \ldots, \bar{c}_e$ are pairwise relatively prime. 
\end{remark}

In the numerical semigroup case one can show that, if $\bar{c}_i = c_i$ for every $i \in \{2, \ldots, e\}$, then $S$ is free. However, the converse is not true. For instance, consider the numerical semigroup $S = \langle 24, 36, 26, 39\rangle$, which is studied in Example \ref{ex:isolated}. In \cite[Proposition 9.15]{ns} both statements are said to be equivalent but there is a mistake in the proof. Therefore, an interesting question is, given a free numerical semigroup $S$, when does $c_i$ equal $\bar{c}_i$ for every $i \in \{2, \ldots, e\}$? In Theorem \ref{thm:alpha:c} we show that $\alpha$-rectangular semigroups have this property.

\section{Isolated factorizations} \label{sec:isolated}

In this section we delve into the concept of isolated factorizations. First of all, we recall the definition and introduce some new notation. Then, we characterize these factorizations, and we study the relation between Betti elements and isolated factorizations. For the particular case of numerical semigroups, we provide several bounds on the number of isolated factorizations and prove that they are sharp. Finally, we apply our results to free and complete intersection numerical semigroups.

\subsection{Definitions and notation}

Let $M$ be a monoid with the ascending chain condition on principal ideals, and let $m \in M$. Recall that a factorization of $m$ is \emph{isolated} if it is disjoint with any other factorization of $s$ or, equivalently, its $\mathrm{R}$-class is a singleton. We denote the set of the isolated factorizations of $m$ by $\mathrm{I}(m)$ and we define $\mathrm{I}(M) = \bigcup_{s \in M} \mathrm{I}(m)$. We refer to the cardinality of $\mathrm{I}(m)$ and $\mathrm{I}(M)$ as $\mathrm{i}(m)$ and $\mathrm{i}(M)$, respectively. 

Recall that if an element of $M$ has an isolated factorization, then either it is a Betti element or it only has one factorization. In light of this fact, we define the sets 
\[\mathrm{I}_s(M) = \{x \in \mathrm{Z}(m) \colon m \in M, \mathfrak{d}(m) = 1\}\] 
and 
\[\mathrm{I}_b(M) = \{x \in \mathrm{I}(m) \colon m \in \mathrm{Betti}(M)\}.\] Let $\mathrm{i}_s(M)$ and $\mathrm{i}_b(M)$ be the cardinalities of the sets $\mathrm{I}_s(M)$ and $\mathrm{I}_b(M)$, respectively. 
With this notation,
\[\mathrm{I}(M) = \mathrm{I}_s(M) \cup \mathrm{I}_b(M).\] 
Note that the sets $\mathrm{I}_s(M)$ and $\mathrm{I}_b(M)$ are disjoint and, thus, $\mathrm{i}(M) = \mathrm{i}_s(M) + \mathrm{i}_b(M)$.

\subsection{A characterization}

Let $M$ be a monoid with the ascending chain condition on principal ideals, and let $A$ be its set of atoms. We use $\le $ to denote the usual partial ordering on $\mathbb{N}^{(A)}$, that is $(x_a)_{a\in A}\le (y_a)_{a\in A}$ if $x_a\le y_a$ for all $a\in A$. As usual, $x<y$ denotes  $x\le y$ and $x\ne y$.
 
Non isolated factorizations can be characterized as follows. 

\begin{lemma} \label{lem:isolated}
  Let $M$ be a monoid with the ascending chain condition on principal ideals, and let $m \in M$. A factorization $x \in \mathrm{Z}(m)$ is not isolated if and only if there exists $z \in \mathrm{Z}(\mathrm{Betti}(M))$ such that $z < x$.
\end{lemma}
\begin{proof}
  Let $\rho$ be a minimal presentation of $M$, and $A=\mathcal{A}(M)$. Set $x \in \mathrm{Z}(m)$. If $x$ is not isolated, then there exists $y \in \mathrm{Z}(m)$ with $x \cdot y \ne 0$.  We can write $x = x' + c$ and $y = y' + c$ with $c, x', y' \in \mathbb{N}^
  {(A)}$ and $x' \cdot y' = 0$. Note that $x' < x$ and $(x', y') \in \mathrm{cong}(\rho)$. There exist $x_0, x_1, \ldots, x_l \in \mathrm{Z}(\varphi(x'))$ such that $x_0 = x'$, $x_l = y'$ and $(x_{i-1}, x_i) \in \rho^1$ for every $i\in \{1, \ldots, l\}$. In particular, we find $u \in \mathbb{N}^{(A)}$ such that $(x' - u, x_1 - u) \in \rho^0$. We have $b = \varphi(x'-u) \in \mathrm{Betti}(M)$ and $z = x' - u < x$. 

  Now let us assume that there are $b \in \mathrm{Betti}(M)$ and $z \in \mathrm{Z}(b)$ such that $z < x$. We can choose $z' \in \mathrm{Z}(b) \setminus \{z\}$ and set $y = (x-z) + z' \in \mathrm{Z}(m)$. Note that $y \ne x$ but $x \cdot y \ne 0$. Therefore, $x$ is not an isolated factorization.
\end{proof}

The first implication of Lemma \ref{lem:isolated} can be strengthened as follows.

\begin{lemma} \label{lem:isolated:2}
  Let $M$ be a monoid with the ascending chain condition on principal ideals, and let $m$ be an element of $M$. If a factorization $x \in \mathrm{Z}(m)$ is not isolated, then there exists $z \in \mathrm{I}_b(M)$ such that $z < x$. In particular, 
  \[ \mathrm{I}_b(M) = \mathrm{Minimals}_{\le} \mathrm{Z}(\{m\in M : \mathfrak{d}(m) \ge 2\}). \]
\end{lemma}
\begin{proof}
  Let $x \in \mathrm{Z}(m)$ be a non isolated factorization. By applying Lemma \ref{lem:isolated} we find  $b_1 \in \mathrm{Betti}(M)$ and $z_1 \in \mathrm{Z}(b)$ with $z_1 < x$. If $z_1$ is isolated, then we are done. Otherwise, we can apply Lemma \ref{lem:isolated} to $z_1$, obtaining $b_2 \in \mathrm{Betti}(M)$ and $z_2 \in \mathrm{Z}(b_2)$ with $z_2 < z_1$. We can repeat the process until we find an isolated factorization $z_i$. Observe that $M$ has the ascending chain condition on principal ideals and, consequently, the sequence $b_1>_M\dots >_M b_i>_M \dots$ will stop in a finite number of steps. 
  
  Assume that $\mathfrak{d}(m) \ge 2$ and let $x$ be a factorization of $m$. If $x$ is not isolated, then we can apply the first assertion to find $b \in \mathrm{Betti}(M)$ and $z \in \mathrm{I}(b)$ such that $z < x$, that is, $x$ is not minimal. Therefore, if $x$ is minimal, then it must be isolated and $x \in \mathrm{I}_b(M)$.  

  Now let $x\in \mathrm{I}_b(M)$. Then $m=\varphi(x)$ has two or more factorizations. If there is $z<x$ with $z$ in $\mathrm{Z}(\{m\in M : \mathfrak{d}(m) \ge 2\})$, then taking $y=x-z$ and $z'\in\varphi^{-1}(z)\setminus\{z\}$, we would get that $x'=y+z'$ is another factorization of $m$ and $x\cdot x'\neq 0$, contradicting that $x$ was isolated.
\end{proof}

\begin{example} \label{ex:isolated:e2}
  Let $S$ be numerical semigroup with embedding dimension $2$. Let $n_1$ and $n_2$ be the minimal generators of $S$. The fact that $\mathrm{Betti}(S) = \{n_1 n_2\}$ is a classical result. Note that $\mathrm{Z}(n_1 n_2) = \{n_1 \mathbf{e}_1, n_2 \mathbf{e}_2\}$. By applying Lemma \ref{lem:isolated} we obtain $\{s \in S \colon \mathfrak{d}(s) = 1\} =  \mathrm{Ap}(S; n_1 n_2)$. Hence, we have $\mathrm{i}_s(S) = n_1 n_2$ and $\mathrm{i}_b(S) = 2$.
\end{example}


\subsection{Isolated factorizations and Betti elements}

We denote the set of the Betti elements of $M$ that have isolated factorizations by $\mathrm{IBetti}(M)$. At this point it is natural to ask whether every Betti element has an isolated factorization or not. However, the answer is negative as the following example shows; this example was obtained by using the \texttt{GAP} \cite{gap} package \texttt{numericalsgps} \cite{numericalsgps}. Under some restrictive hypothesis one can show that $\mathrm{IBetti}(M) = \mathrm{Betti}(M)$, see Corollary \ref{cor:alpha:betti} and Theorem \ref{thm:ci-b1}. This will be the case of Betti sorted numerical semigroups.

\begin{example} \label{ex:isolated}
 The numerical semigroup $S = \langle 24, 26, 36, 39\rangle$ has a Betti element without isolated factorizations. We have $\mathrm{Betti}(S) = \{72, 78, 156\}$. Furthermore, these elements have the following factorizations:
  \begin{itemize}
  \item $\mathrm{Z}(72) = \{(3,0,0,0), (0,0,2,0)\}$, which are isolated;
  \item $\mathrm{Z}(78) = \{(0,3,0,0), (0,0,0,2)\}$, which are isolated;
  \item $\mathrm{Z}(156) = \{(0,6,0,0), (0,0,0,4), (0,3,0,2), (5,0,1,0), (2,0,3,0)\}$, none of which is isolated.
  \end{itemize}
  Note that $S$ is a complete intersection numerical semigroup. More concretely, the set 
  \[\{(2 \mathbf{e}_3, 3 \mathbf{e}_1), (6 \mathbf{e}_2, 2 \mathbf{e}_1+3 \mathbf{e}_3), (2 \mathbf{e}_4, 3 \mathbf{e}_2)\}\] 
  is a minimal presentation of $S$ and, thus, it is free for the arrangement $(24, 36, 26, 39)$. Recall that $\mathbf{e}_i$ denotes the $i$th column of the $3\times 3$ identity matrix.
\end{example}

As a direct consequence of Lemma \ref{lem:isolated} and Lemma \ref{lem:isolated:2}, one can easily characterize those elements which have a non isolated factorization.

\begin{corollary} \label{cor:isolated}
Let $M$ be a monoid with the ascending chain condition on principal ideals. Let $m$ be an element of $M$. The following statements are equivalent:
\begin{enumerate}
\item $m$ has a non isolated factorization;
\item there exists $b \in \mathrm{IBetti}(M)$ such that $b <_M m$;
\item there exists $b \in \mathrm{Betti}(M)$ such that $b <_M m$.
\end{enumerate}
\end{corollary}

Therefore, if $\mathfrak{d}(m) \ge 2$ but $m$ is not a Betti element, then there is $b \in \mathrm{Betti}(M)$ such that $b <_S m$.  This fact was stated in \cite[Lemma 1]{monoids:up} for finitely generated monoids. Indeed, we can find $b \in \mathrm{IBetti}(M)$ such that $b <_M m$. 

Other subsets of Betti elements have been studied in the literature. According to \cite{minimal-multi}, an element $b$ of $M$ is a \emph{minimal multi-element} if $\mathfrak{d}(b) \ge 2$ but $\mathfrak{d}(b - a) = 1$ for every atom $a$ of $M$ such that $a \le_M b$. One can prove that these elements are in $\mathrm{Betti}(M)$ using the graph $G_b$ described in \cite[Chapter 9]{monoids}  (or \cite[Chapter 7]{ns} for numerical semigroups). Nonetheless, we follow another approach in Proposition \ref{prop:all-i}. 

In \cite{monoids:up} the authors introduce the notion of \emph{Betti-minimal} elements, which are the minimal elements of $\mathrm{Betti}(M)$ with respect to $\le_M$. They denote the set of Betti-minimal elements by $\mathrm{Betti}\text{-}\mathrm{minimals}(M)$. Moreover, they characterize these elements as those with more than one factorization such that all their R-classes are singletons, see \cite[Proposition 3]{monoids:up}. In Proposition \ref{prop:all-i} we recapitulate their result under our own notation and show that it is a consequence of Corollary \ref{cor:isolated}. We also prove that minimal multi-elements and Betti-minimal elements coincide. 

\begin{proposition} \label{prop:all-i}
  Let $M$ be a monoid with the ascending chain condition on principal ideals. Let $m\in M$. The following statements are equivalent:
  \begin{enumerate}
  \item \label{item:betti-minimal} $m$ is Betti-minimal;
  \item \label{item:ibetti-minimal} $m$ is a minimal element of $\mathrm{IBetti}(M)$ with respect to $\le_M$;
  \item \label{item:isolated} $m$ has at least two factorizations and all of them are isolated, that is, $\mathrm{nc}(\nabla_m) = \mathrm{i}(m) \ge 2$;
  \item \label{item:minimal-multi} $m$ is a minimal multi-element.
  \end{enumerate}
\end{proposition}
\begin{proof}
  In light of Corollary \ref{cor:isolated}, we find that a Betti element is minimal in $\mathrm{Betti}(M)$ or $\mathrm{IBetti}(M)$ with respect to $\le_M$ if and only if all its factorizations are isolated. From this fact one sees that \ref{item:betti-minimal}, \ref{item:ibetti-minimal} and \ref{item:isolated} are equivalent. 

  \begin{enumerate}[leftmargin=10pt]
  \item[\ref{item:betti-minimal}\kern-3pt] implies \ref{item:minimal-multi}. Let $a$ be an atom of $M$ with $a \le_M m$. Note that $m - a$ is not in $\mathrm{Betti}(M)$. If $\mathfrak{d}(m-a) \ge 2$, then, in light of Corollary \ref{cor:isolated}, there is $b \in \mathrm{Betti}(M)$ such that $b <_M m-a <_M m$, a contradiction. Hence, we have $\mathfrak{d}(m-a) = 1$ and $m$ is a minimal multi-element.
  \item[\ref{item:minimal-multi}\kern-3pt] implies \ref{item:betti-minimal}. 
  If $b <_M m$, for some Betti element $b$, then there is an atom $a \in M$ with $b \le_M m-a$ and, thus, we have $\mathfrak{d}(m-a) \ge 2$, a contradiction. \qedhere
  \end{enumerate}
\end{proof}

\begin{remark}
  We have the inclusions
  \[\mathrm{Betti}\text{-}\mathrm{minimals}(M) \subseteq \mathrm{IBetti}(M) \subseteq \mathrm{Betti}(M).\]
   In Theorem \ref{thm:ci-b1} we show that if $M$ is a complete intersection simplicial affine semigroup such that the set $\mathrm{Betti}\text{-}\mathrm{minimals}(M)$ is a singleton, then $\mathrm{IBetti}(M) = \mathrm{Betti}(M)$ and, thus, $\mathrm{Betti}\text{-}\mathrm{minimals}(M)$ differs from $\mathrm{IBetti}(M)$ in general. 
\end{remark}

As a consequence of the results given in this section we obtain the following characterization of those elements with an unique expression.

\begin{corollary} \label{cor:isolated:1} 
  Let $M$ be a monoid with the ascending chain condition on principal ideals.  Then
  \[ \{m \in M \colon \mathfrak{d}(m) = 1\} = \bigcap_{b \in \mathrm{Betti}(M)} \mathrm{Ap}(M; b) = \bigcap_{b \in \mathrm{IBetti}(M)} \mathrm{Ap}(M; b) = \bigcap_{b \in \mathrm{Betti}\text{-}\mathrm{minimals}(M)} \mathrm{Ap}(M; b). \]
\end{corollary}
 
The proof of Theorems \ref{thm:alpha-m+1} and \ref{thm:betti-sorted:alpha} make use of the following observation.

\begin{corollary} \label{cor:ap-b1}
  Let $M$ be a monoid with the ascending chain condition on principal ideals. The following conditions are equivalent:
  \begin{enumerate}
  \item \label{item:ap-b1:b} $\mathrm{Betti}\text{-}\mathrm{minimals}(M) = \{b_1\}$;
  \item \label{item:ap-b1:c} $\Ap(M; b_1) = \{m \in M :\mathfrak{d}(M) = 1\}$.
  \end{enumerate}
\end{corollary}

\begin{remark}\label{rem:ap-b1}
Note that if $S$ is a numerical semigroup, then the second condition of the above lemma is equivalent to $\mathrm{i}_s(M) = b_1$.
\end{remark}

Note that if $S$ is a numerical semigroup, then the element $\min \mathrm{Betti}(S)$ is Betti-minimal and, thus, all its factorizations are isolated. In the proof of \cite[Lemma 15]{cyclotomic} the authors show that $\min \mathrm{Betti}(S)$ is the smallest element of $S$ with more than one factorization. We collect all this information in the following result, where we give a simpler proof that uses isolated factorizations.

\begin{corollary} \label{cor:b1}
  Let $S$ be a numerical semigroup. Then, $\min \mathrm{Betti}(S)$ is the smallest element of $S$ with more than one factorization. Moreover, all its factorizations are isolated.
\end{corollary}
\begin{proof}
  Let $d$ be the smallest element of $S$ with more than one factorization and set $b_1 = \min \mathrm{Betti}(S)$. Note that $d \le b_1$. In light of Corollary \ref{cor:isolated}, there is $b \in \mathrm{IBetti}(S)$ with $b \le_S d$ and, in particular, we have $b \le d \le b_1$. The only possibility is $b = d = b_1$. \qedhere
\end{proof}

A natural question is whether isolated factorizations of different elements are disjoint. Of course the answer is negative. For example, the elements $\mathrm{m}(S)$ and $2\mathrm{m}(S)$ in a numerical semigroup $S$ have only one factorization but these factorizations are not disjoint. However, if we focus on Betti elements, then the following can be shown.

\begin{lemma}\label{lem:disjoint-betti}
  Let $M$ be a monoid with the ascending chain condition on principal ideals. Let $b_1$ and $b_2$ be two Betti elements of $M$ such that $b_1 <_M b_2$.
 Then, $x \cdot y = 0$ for every $x \in \mathrm{Z}(b_1)$ and $y \in \mathrm{I}(b_2)$. 
\end{lemma}
\begin{proof}
  Let us suppose that $\mathrm{I}(b_2) \ne \emptyset$. Set $x, x' \in \mathrm{Z}(b_1)$ with $x \ne x'$ and $y\in \mathrm{I}(b_2)$. Fix $w \in \mathrm{Z}(b_2-b_1)$ and define $z = w+x$ and $z' = w+x'$. We have $z, z' \in \mathrm{Z}(b_2)$ and $z \cdot z' \ne 0$. Hence, $z$ is not isolated and $y \ne z$. Consequently, we find that $y \cdot z = 0$ and, thus, $x \cdot y = 0$.
\end{proof}

In general isolated factorizations of different Betti elements are not disjoint as the following example shows. We will exploit Lemma \ref{lem:disjoint-betti} when we consider Betti sorted monoids.

\begin{example} \label{ex:disjoint}
Set $S = \langle 4, 5, 6\rangle$. We have $\mathrm{Z}(12) = \{(1,0,1), (0,2,0)\}$ and $\mathrm{Z}(10) = \{(3,0,0), (0,0,2)\}$. One can show that $\mathrm{Betti}(S) = \{10, 12\}$. The factorization $(0,2,0)$ is disjoint with the factorizations of $10$ but $(1,0,1)$ is not. 
\end{example}

In the rest of this section we determine several isolated factorizations of $M$. The obtained results will be useful to bound the number of isolated factorizations of some semigroups. To prove Theorem \ref{thm:isolated} we use the following lemma, which is folklore; we include it here for sake of completeness.

\begin{lemma} \label{lem:ca}
  Let $M$ be a monoid with the ascending chain condition on principal ideals. Let $a$ be an atom of $M$ such that the set $\{c \in \mathbb{Z}^+ : c a \in \langle \mathcal{A}(M) \setminus \{a\} \rangle\}$ is not empty and let $c_a$ be its minimum. Then, $c_a \mathbf{e}_a$ is an isolated factorization of $M$ and $c_a a$ is a Betti element. Moreover, the elements $a, 2a, \ldots, (c_a-1)a$ have unique expressions.
\end{lemma}
\begin{proof}
  Let $k \in \mathbb{Z}^+$ such that $\mathfrak{d}(k a) \ge 2$. In view of Lemma \ref{lem:isolated:2}, there is $b \in \mathrm{IBetti}(S)$ and $x \in \mathrm{I}(b)$ such that $x \le k \mathbf{e}_a$. Hence, $x$ equals $c \mathbf{e}_a$ for some $c \in \mathbb{Z}^+$ with $c \le k$. Since $\mathfrak{d}(b) \ge 2$ and $x$ is isolated, there is $y \in \mathrm{Z}(b)$ such that $x \cdot y = 0$. Consequently, we have $c a = b \in \langle \mathcal{A}(M) \setminus \{a\} \rangle$ and $k \ge c \ge c_a$. In particular,  for $k = c_a$ we find that $c_a \mathbf{e}_a$ is an isolated factorization. 
\end{proof}

In light of the previous lemma, we define $\mathcal{C}(M)$ as the set of atoms $a\in \mathcal{A}(M)$ for which $\{c \in \mathbb{Z}^+ : c a \in \langle \mathcal{A}(M) \setminus \{a\} \rangle\}$ is not empty, $c_a$ being the minimum of this set. 

\begin{theorem} \label{thm:isolated}
  Let $M$ be a monoid with the ascending chain condition on principal ideals. Then, we have
  \[\{c_a \mathbf{e}_a \colon a \in \mathcal{C}(M)\} \subseteq \mathrm{I}_b(S)\] 
  and
  \[\mathrm{I}(S) \setminus \{c_a \mathbf{e}_a \colon a\in \mathcal{C}(M)\}  \subseteq \left\{\sum\nolimits_{a \in I} \lambda_a \mathbf{e}_a \colon I\subseteq \mathcal{A}(M), I\hbox{ finite}, \lambda_a\in \mathbb{N} \text{ and}, \lambda_a<c_a \hbox{ for }a\in \mathcal{C}(M)\cap I\right\}. \]
  Moreover, if one of the inclusions is an equality, then the other one is also an equality.
\end{theorem}

\begin{proof}
  The first inclusion is a consequence of Lemma \ref{lem:ca}. 
  
  In order to prove the second inclusion, let $x$ be an isolated factorization with $x \not \in \{c_a \mathbf{e}_a \colon a \in \mathcal{C}(M)\}$. Write $x=\sum_{a\in I} \lambda_a \mathbf{e}_a$ with $I$ a finite subset of $\mathcal{A}(M)$ and $\lambda_a$ a positive integer. If $\lambda_a\ge c_a$, for some $a\in \mathcal{C}(M)$, then $c_a\mathbf{e}_a<x$, which in light of Lemma \ref{lem:isolated} contradicts that $x$ is an isolated factorization. Hence, $\lambda_a<c_a$ for all $a\in \mathcal{C}(M)$. 
  
  Now let us assume that $\mathrm{I}_b(S) = \{c_a \mathbf{e}_a \colon a \in \mathcal{C}(M) \}$. Let us consider a factorization $x = \sum\nolimits_{a \in I} \lambda_a \mathbf{e}_a$ with $I$ a finite subset of $\mathcal{A}(M)$, and $0 \le \lambda_a < c_a$ for every $a\in I\cap \mathcal{C}(M)$. Note that for any $y \in \mathrm{I}_b(S)$, $y$ is not smaller than $x$. Hence, we obtain that $x$ is isolated by invoking Lemma \ref{lem:isolated:2}. Since $x$ is not an element of $\mathrm{I}_b(S)$, we find that $x \in \mathrm{I}_s(S) = \mathrm{I}(S) \setminus \mathrm{I}_b(S)$.
  
  Finally, suppose that the second inclusion is an equality. Set $A = \mathrm{I}(M) \setminus \{c_a \mathbf{e}_a \colon a\in \mathcal{C}(M)\}$ and $x \in A$. Note that $\mathrm{I}_s(M)$ is contained in $A$. Let us assume that $x \in \mathrm{I}_b(S)$ in order to obtain a contradiction. Then, by applying Lemma \ref{lem:isolated}, there is no $y\in \mathrm{I}(M)$ such that $x<y$. Thus, $x$ is maximal in $A$ with respect to the usual partial ordering. The factorization $x$ is of the form $x=\sum_{a\in I} \lambda_a \mathbf{e}_a$, with $I$ a finite subset of $\mathcal{A}(M)$ and $0\le \lambda_a<c_a$ for $a\in I\cap \mathcal{C}(M)$. If $\lambda_a=0$ for some $a\in \mathcal{A}(M)$, then $x+\mathbf{e}_a\in A$, contradicting the maximality of $x$. Thus $\lambda_a\neq 0$ for all $a\in \mathcal{A}(M)$, and this factorization is not isolated (it has common support with any other factorization), contradicting that $x\in \mathrm{I}_b(M)$.
\end{proof}


\begin{example}
    Let $M=\langle (1,0),(0,2),(0,3)\rangle\subset \mathbb{N}^2$. For this monoid we have $\mathcal{C}(M)=\{(0,2),(0,3)\}$ and $\mathrm{I}_b(M)=\{(0,3,0),(0,0,2)\}$. On the other hand, we have $\mathrm{I}_s(M)=\{(x,y,z)\colon y\le 2, z\le 1\}$. Therefore, we cannot expect $\mathcal{C}(M)$ to be equal to $\mathcal{A}(M)$.
\end{example}

\begin{example}
    It may happen also that $\mathcal{C}(M)$ is empty. Take for instance 
    \[M=\langle (1,0,1),(0,1,0),(1,1,0),(0,0,1)\rangle.\] 
    In this setting $\mathrm{Betti}(M)=\{(1,1,1)\}$.
\end{example}

\begin{remark} \label{rem:isolated}
  If $S$ is a numerical semigroup minimally generated by $\{n_1, \ldots, n_e\}$, then we note that $\mathcal{C}(S) = \{n_1, \ldots, n_e\}$. Therefore, Theorem \ref{thm:isolated} states that $\{c_i \mathbf{e}_i \colon i \in \{1, \ldots, e\}\} \subseteq \mathrm{I}_b(S)$ and
  \[\mathrm{I}(S) \setminus \{c_i \mathbf{e}_i \colon i \in \{1, \ldots, e\}\}  \subseteq \left\{\sum\nolimits_{i = 1}^e \lambda_i \mathbf{e}_i \colon 0 \le \lambda_i < c_i\right\}. \]
  Moreover, if one of the inclusions is an equality, then the other one becomes an equality.
\end{remark}

\subsection{Bounds on the number of isolated factorizations of simplicial affine semigroups}

In this section we bound the number of isolated factorizations for numerical semigroups and, when possible, simplicial affine semigroups. 

The following lower bound is attained by the semigroups studied in Section \ref{sec:ci}.

\begin{lemma} \label{lem:simplicial:bound}
  Let $S$ be a simplicial affine semigroup of $\mathbb{N}^r$ minimally generated by $\{n_1, \ldots, n_{r+m}\}$. Then, there is $x \in \langle \mathbf{e}_1, \ldots, \mathbf{e}_r \rangle$ such that $x \in \mathrm{I}_b(S)$. Moreover, we have $\mathrm{i}_b(S) \ge m+1$. 
\end{lemma}
\begin{proof}
  There are non negative integers $\lambda_{1}, \ldots, \lambda_{r}$ such that $c^*_{r+1} n_{r+1} = \lambda_1 n_1 + \cdots + \lambda_r n_r$. We consider the factorization $y = \lambda_1 \mathbf{e}_1 + \cdots + \lambda_r \mathbf{e}_r$. 
  In light of Lemma \ref{lem:isolated:2}, there is $x \in \mathrm{I}_b(S)$ such that $x \le y$. Hence, we have $x \in \langle \mathbf{e}_1, \ldots, \mathbf{e}_r \rangle$. Finally, recall that $n_{r+1}, \ldots, n_{r+m} \in \mathcal{C}(M)$. Therefore, $x \in \mathrm{I}_b(S)$, and  Theorem \ref{thm:isolated} ensures that $c_{r+1} \mathbf{e}_{r+1}, \ldots, c_{r+m} \mathbf{e}_{r+m} \in \mathrm{I}_b(S)$ . Thus,  $\mathrm{i}_b(S) \ge m+1$.
\end{proof}

Now we give an upper bound for $\mathrm{i}_b(S)$ when $S$ is Cohen-Macaulay. First, let us recall some concepts. Let $S$ be a Cohen-Macaulay simplicial semigroup of $\mathbb{N}^r$ minimally generated by $\{n_1, \ldots, n_{r+m}\}$. The cardinality of a minimal presentation of $S$ is upper bounded by $(2d -m) (m-1)/2 + 1$, where $d = \#\Ap(S; \{n_1, \ldots, n_r\})$ \cite[Theorem 2.6]{cm}. The inequality $m \le \# d-1$ is a folklore result. One can show that $(2d -m) (m-1)/2 + 1 \le d(d-1)/2$ for every $m \le d-1$. In \cite{cm} the authors also prove that the cardinality of a minimal presentation of $S$ equals $d(d-1)/2$ if and only if $S$ has \emph{maximal codimension}, that is, $m = d-1$. 

\begin{corollary} \label{cor:iso:bound:b}
  Let $S$ be a Cohen-Macaulay and simplicial affine semigroup of $\mathbb{N}^r$ minimally generated by $\{n_1, \ldots, n_{r+m}\}$. Let $d = \#\Ap(S; \{n_1, \ldots, n_r\})$. Then
  \[ \mathrm{i}_b(S) \le \sum_{b \in \mathrm{Betti}(S)} \mathrm{nc}(\nabla_{b}) \le (2d -m) (m-1) + 2 \le d (d -1).\]
  Furthermore, $\mathrm{i}_b(S) = d (d -1)$ if and only if $S$ has maximal codimension and $\mathfrak{d}(b) = 2$ for every $b \in \mathrm{Betti}(S)$.
\end{corollary}
\begin{proof}
Note that $\mathrm{i}_b(S) = \sum_{b \in \mathrm{Betti}(S)} \mathrm{i}(b)$ and $\mathrm{i}(b) \le \mathrm{nc}(\nabla_b)$  for every $b \in \mathrm{Betti}(S)$ . From these observations one obtain the first inequality, which is reached if and only if $\mathrm{i}(b) = \mathrm{nc}(\nabla_b)$ for every $b \in \mathrm{Betti}(S)$. As a consequence, we have
  \[\mathrm{i}_b(S) \le \sum_{b \in\mathrm{Betti}(S)} \mathrm{nc}(\nabla_{b}) \le 2 \sum_{b \in\mathrm{Betti}(S)}(\mathrm{nc}(\nabla_b)-1) \le (2d -m) (m-1) + 2 \le d (d -1).\]
  Note that $\mathrm{i}_b(S) = d (d -1)$ if and only if all the inequalities applied are attained. Moreover, the equality 
  \[\sum_{b \in\mathrm{Betti}(S)} \mathrm{nc}(\nabla_{b}) = 2 \sum_{b \in\mathrm{Betti}(S)}(\mathrm{nc}(\nabla_b)-1)\]
  holds if and only if \[\sum_{b \in\mathrm{Betti}(S)} 1 = \sum_{b \in\mathrm{Betti}(S)}(\mathrm{nc}(\nabla_b)-1),\]
  that is, $\mathrm{nc}(\nabla_b) = 2$ for every $b \in \mathrm{Betti}(S)$. Set $b \in \mathrm{Betti}(S)$. The proof is completed by noticing that $\mathfrak{d}(b) = 2$  if and only if $\mathrm{i}(b) = \mathrm{nc}(\nabla_b) = 2$.
\end{proof}

The lower and the upper bound of $\mathrm{i}_b(S)$ may hold at the same time as the following example shows.

\begin{example}
  We look for all the numerical semigroups which verify $\mathrm{e}(S) = \mathrm{i}_b(S) = \mathrm{m}(S) (\mathrm{m}(S) -1)$. Since $\mathrm{e}(S) \le \mathrm{m}(S)$, we have $\mathrm{m}(S) = \mathrm{m}(S) (\mathrm{m}(S) -1)$ and, thus, $\mathrm{e}(S) = \mathrm{m}(S) = 2$. Consequently, the solutions are the numerical semigroups generated by $\{2, n\}$, where $n$ is an odd integer.
\end{example}


\begin{remark}
If we remove the Cohen-Macaulay condition in Corollary \ref{cor:iso:bound:b}, then we can apply \cite[Theorem 4.1]{simplicial-structure} and obtain $\mathrm{i}_b(S) \le 2d(d-1)$.
\end{remark}

Now we try to bound $\mathrm{i}_s(S)$. Let $S'$ be a numerical semigroup and consider the semigroup $S= \mathbb{N}^{r-1}\times S'$, which is a simplicial semigroup of $\mathbb{N}^r$ with codimension $\mathrm{e}(S')-1$. Note that the set $\mathrm{I}_s(S)$ is not finite. However, if we assume $S$ to be a numerical semigroup, then the fact that its Ap\'ery sets are finite allows us to find the following bounds.

\begin{corollary} \label{cor:iso:bound:d}
  Let $S$ be a numerical semigroup. Then
  \[\mathrm{e}(S) + 3 \le \sum_{i = 1}^{\mathrm{e}(S)} c_i - \mathrm{e}(S) + 2 \le \mathrm{i}_s(S) \le \min \mathrm{Betti}(S).\]
\end{corollary}
\begin{proof}
  Let $\{n_1 < \cdots < n_e\}$ be the minimal system of generators of $S$. Note that $0$, the minimal generators, $2 n_1$ and $n_1 + n_2$ only have one factorization. Thus, we have $c_1 \ge 3$. These observations in conjunction with Lemma \ref{lem:ca} yield the inequalities
  \[ \mathrm{e}(S) + 3 \le \sum_{i = 1}^{\mathrm{e}(S)} (c_i-1) + 2  =  \sum_{i = 1}^{\mathrm{e}(S)} c_i - \mathrm{e}(S) + 2 \le  \mathrm{i}_s(S). \]
  The upper bound is a consequence of Corollary \ref{cor:isolated:1} and the fact that the cardinality of $\Ap(S;n)$ is $n$ for every $n\in S\setminus \{0\}$.
\end{proof}

We wonder whether the bounds given in Corollary \ref{cor:iso:bound:d} can be attained. In Example \ref{ex:isolated:e2} we showed that the upper bound is attained for numerical semigroups with embedding dimension $2$. Indeed, Corollary \ref{cor:ap-b1} characterizes those numerical semigroups such that the upper bound of Corollary \ref{cor:iso:bound:d} is an equality (see Remark \ref{rem:ap-b1}).

\begin{example}
  Let us consider the numerical semigroup $S = \langle 3, 4, 5 \rangle$. Corollary \ref{cor:isolated:1} provides us with an algorithm to compute the set $\mathrm{I}_s(S)$. One can use the package \texttt{numericalsgps} \cite{numericalsgps} to perform the computations. For this semigroup we have 
  \[\mathrm{I}_s(S) = \{0, \mathbf{e}_1, \mathbf{e}_2, \mathbf{e}_3, 2 \mathbf{e}_1, \mathbf{e}_1+\mathbf{e}_2\}\]
  and, thus, the lower bounds of Corollary \ref{cor:iso:bound:d} are reached.
\end{example}

 Finally, we provide other bounds in the next result, which is verified by numerical semigroups.

\begin{corollary} \label{cor:iso:bound}
  Let $M$ be a finitely generated monoid with the ascending chain condition on principal ideals such that $\mathcal{A}(M) = \mathcal{C}(M)$. Let $e = \# \mathcal{A}(M)$. Then, we have
  \[\mathrm{i}(M) \le e + \prod_{a \in \mathcal{A}(M)} c_a \quad \text{ and }  \quad \mathrm{i}_s(m) \le \prod_{a \in \mathcal{A}(M)} c_a.\]
  Moreover, the following statements are equivalent:
  \begin{enumerate}
    \item $\mathrm{i}(M) = \mathrm{e}(S) + \prod_{a \in \mathcal{A}(M)} c_a$;
    \item $\mathrm{i}_b(M) = e$;
    \item $\mathrm{i}_s(M) = \prod_{a \in \mathcal{A}(M)} c_a$.
  \end{enumerate}
\end{corollary}
\begin{proof}
  This is a consequence of Theorem \ref{thm:isolated} and the equality 
  \[\#\left\{\sum\nolimits_{a \in \mathcal{A}(M)} \lambda_a \mathbf{e}_a \colon 0 \le \lambda_a < c_a\right\} = \prod_{a \in \mathcal{A}(M)} c_a. \qedhere \] 
\end{proof}


\begin{remark}
  Let $S$ be a numerical semigroup. Recall that $\mathrm{i}(S) = \mathrm{i}_s(S) + \mathrm{i}_b(S)$. Consequently, some of the previous results can be mixed to obtain
  \[2\mathrm{e}(S) + 3 \le \mathrm{i}(S) \le \min \mathrm{Betti}(S) + \sum_{b \in \mathrm{Betti}(S)} \mathrm{nc}(\nabla_{b}) \le \min \mathrm{Betti}(S) +  \mathrm{m}(S) (\mathrm{m}(S) -1).\]
\end{remark}

Now we study when the previous upper bound is attained.

\begin{corollary} \label{cor:iso:one-betti}
  Let $S$ be a numerical semigroup. The following conditions are equivalent:
  \begin{itemize}
      \item the set $\mathrm{Betti}(S)$ is a singleton;
      \item $\mathrm{i}(S) = \min \mathrm{Betti}(S) + \sum_{b \in \mathrm{Betti}(S)} \mathrm{nc}(\nabla_{b})$.
  \end{itemize}
\end{corollary}
\begin{proof}
  Note that the $\mathrm{i}(S) = \min \mathrm{Betti}(S) + \sum_{b \in \mathrm{Betti}(S)} \mathrm{nc}(\nabla_{b})$ if and only if $\mathrm{i}_s(S) = \min \mathrm{Betti}(S)$ and $\mathrm{i}_b(S) =  \sum_{b \in \mathrm{Betti}(S)} \mathrm{nc}(\nabla_{b})$. These equalities hold at the same time if and only if $\mathrm{Betti}\text{-}\mathrm{minimals}(S)$ is singleton (Corollary \ref{cor:ap-b1}) and it equals $\mathrm{Betti}(S)$ (Proposition \ref{prop:all-i}).
\end{proof}

\begin{example}
  Let us find all numerical semigroups with
  \[\mathrm{i}(S) = \min \mathrm{Betti}(S) + \mathrm{m}(S) (\mathrm{m}(S) -1),\]
  that is, $\mathrm{i}_s(S) = \min \mathrm{Betti}(S)$ and $\mathrm{i}_b(S) = \mathrm{m}(S) (\mathrm{m}(S) -1)$. Note that these equalities hold at the same time if and only if $\mathrm{Betti}(S) = \{b_1\}$, $\mathfrak{d}(b_1) = 2$ and $\mathrm{e}(S) = \mathrm{m}(S)$. If it is the case, then the cardinality of a minimal presentation is $1$, which is greater or equal than $\mathrm{e}(S)-1$. The only possibility is $\mathrm{e}(S) = 2$.  Hence, again the solutions are those semigroups whose embedding dimension and multiplicity are $2$.
  
\end{example}

\section{$\alpha$-rectangular semigroups} \label{sec:alpha}

Let $S$ be a simplicial affine semigroup of $\mathbb{N}^r$ minimally generated by $\{n_1, \ldots, n_{r+m}\}$. We say that $S$ is \emph{rectangular for $\{n_1, \ldots, n_r\}$} if its Ap\'ery set with respect to $\{n_1, \ldots, n_r\}$ is of the form 
\[\Ap(S;\{n_1,\ldots,n_r\}) = \left\{\sum\nolimits_{i = r+1}^{r+m} \lambda_i n_i \colon 0 \le \lambda_i \le \mu_i \right\}\]
for some non negative integers $\mu_{r+1}, \ldots, \mu_{r+m}$. We say that $S$ is \emph{rectangular} if it is rectangular for some minimal generators $\{n_1, \ldots, n_r\}$ with $\mathrm{L}_{\mathbb{Q}^+}(S) = \mathrm{L}_{\mathbb{Q}^+}(\{n_1, \ldots, n_r\})$. If it is the case, then $\Ap(S; \{n_1, \ldots, n_r\})$ has a unique maximal element with respect to the semigroup order. Therefore, if $S$ is Cohen-Macaulay, then $S$ is Gorenstein.

As a consequence of Lemma \ref{lem:free}, if $S$ is free for the arrangement $(n_1, \ldots, n_{r+m})$, then it is rectangular for $\{n_1, \ldots, n_r\}$ and $\mu_i = c_i^*-1$ for every $i \in \{r+1, \ldots, r+m\}$.

\begin{corollary} \label{cor:rectangular}
Let $S$ be a simplicial affine semigroup of $\mathbb{N}^r$  minimally generated by $\{n_1, \ldots, n_{r+m}\}$. If the equality
\[\Ap(S; \{n_1, \ldots, n_r\}) = \left\{\sum\nolimits_{i = r+1}^{r+m} \lambda_i n_i \colon 0 \le \lambda_i \le \mu_i \right\}\]
holds for some non negative integers $\mu_{r+1}, \ldots, \mu_{r+m}$, then $c_i \le \mu_i +1$ for every $i \in \{r+1, \ldots, r+m\}$.
\end{corollary}
\begin{proof}
  Let $i \in \{r+1, \ldots, r+m\}$. In light of Lemma \ref{lem:ca}, $(c_i-1) n_i$ has a unique expression and, thus, $(c_i -1) n_i \in \Ap(S; \{n_1, \ldots, n_r\})$. Hence, we can write $(c_i - 1) n_i = \sum_{j = r+1}^{r+m} \lambda_j n_j$ for some integers $\lambda_{r+1}, \ldots, \lambda_{r+m}$ such that $0 \le \lambda_j \le \mu_j$ for every $j \in \{r+1, \ldots, r+m\}$. Again, since $(c_i-1) n_i$ has a unique expression we find that $\lambda_j = c_i-1$ when $j = i$ and $\lambda_j = 0$ otherwise. In particular, we obtain $c_i - 1 = \lambda_i \le \mu_i$.
\end{proof}

We say that $S$ is \emph{$c$-rectangular for $\{n_1, \ldots, n_r\}$} if the inequalities given in Corollary \ref{cor:rectangular} are attained, that is,
\[\Ap(S; \{n_1, \ldots, n_r\}) = \left\{\sum\nolimits_{i = r+1}^{r+m} \lambda_i n_i \colon 0 \le \lambda_i < c_i \right\}.\]
The semigroup $S$ is \emph{$c$-rectangular} if it is $c$-rectangular for some minimal generators $\{n_1, \ldots, n_r\}$ such that $\mathrm{L}_{\mathbb{Q}^+}(S) = \mathrm{L}_{\mathbb{Q}^+}(\{n_1, \ldots, n_r\})$. 

In this section we introduce another family of rectangular semigroups which generalizes the numerical semigroups studied in \cite{ci:classes:2}. First, for each $i \in \{r+1, \ldots, r+m\}$ define the constant 
\[\alpha_i = \max \{h \in \mathbb{Z}^+ \colon h n_i \in \Ap(S; \{n_1, \ldots, n_r\})\}.\] A simplicial affine semigroup is \emph{$\alpha$-rectangular for $\{n_1, \ldots, n_r\}$} if its Ap\'ery set is of the form 
\[\Ap(S; \{n_1, \ldots, n_r\}) = \left\{\sum\nolimits_{i = r+1}^{r+m} \lambda_i n_i \colon 0 \le \lambda_i \le \alpha_i \right\}.\]
 The semigroup $S$ is \emph{$\alpha$-rectangular} if it is $\alpha$-rectangular for some minimal generators $\{n_1, \ldots, n_r\}$ with $\mathrm{L}_{\mathbb{Q}^+}(S) = \mathrm{L}_{\mathbb{Q}^+}(\{n_1, \ldots, n_r\})$. 

\begin{remark}
  In \cite{ci:classes:2} the authors define a \emph{$\alpha$-rectangular numerical semigroup} as a numerical semigroup $S$ minimally generated by $\{n_1 < \cdots < n_e\}$ such that 
  \[\Ap(S; n_1) = \left\{\sum\nolimits_{i = 2}^{e} \lambda_i n_i \colon 0 \le \lambda_i \le \alpha_i \right\},\]
  where $\alpha_i = \max \{h \in \mathrm{Z}^+ \colon h n_i \in \Ap(S;n_1)\}$ for every $i \in \{2, \ldots, e\}$. Recall that a numerical semigroup is simplicial for any order of its minimal generators. Therefore, according to our definition, a numerical semigroup is $\alpha$-rectangular if and only if there is a minimal generator $n_j$ such that 
  \[\Ap(S; n_j) = \left\{\sum\nolimits_{i = 1; i \ne j}^{e} \lambda_i n_i \colon 0 \le \lambda_i \le \alpha_i \right\},\]
  where $\alpha_i = \max \{h \in \mathrm{Z}^+ \colon h n_i \in \Ap(S;n_j)\}$ for every $i \in \{1, \ldots, e\}$. If it is the case, then we say that it is \emph{$\alpha$-rectangular} for $n_j$ or the $j^{th}$ minimal generator. 
\end{remark}

Our definition generalizes that of \cite{ci:classes:2}. The characterizations obtained in  \cite{ci:classes:2} for numerical semigroups can be easily generalized to the current setting; Proposition \ref{prop:alpha} and Proposition \ref{prop:alpha:gluing} generalize \cite[Proposition 2.6]{ci:classes:2} and \cite[Theorem 3.3]{ci:classes:2}, respectively. 

The proof of the following lemma is a direct consequence of the definition of $\alpha_i$.

\begin{lemma} \label{lem:alpha:apery}
    Let $S$ be a simplicial affine semigroup of $\mathbb{N}^r$ minimally generated by $\{n_1, \ldots, n_{r+m}\}$. Then, 
    \[ \Ap(S; \{n_1, \ldots, n_r\}) \subseteq \left\{\sum\nolimits_{i = r+1}^{r+m} \lambda_i n_i : 0 \le \lambda_i \le \alpha_i, i \in \{r+1, \ldots, r+m\} \right\}. \]
\end{lemma}

\begin{proposition} \label{prop:alpha}
    Let $S$ be a simplicial affine semigroup of $\mathbb{N}^r$ minimally generated by $\{n_1, \ldots, n_{r+m}\}$. The following statements are equivalent:
  \begin{enumerate}
      \item \label{item:alpha} $S$ is an $\alpha$-rectangular semigroup for $\{n_1, \ldots, n_r\}$;
      \item \label{item:alpha:maximal} there is only one maximal element in $\Ap(S; \{n_1, \ldots, n_r\})$ and it has a unique expression;
      \item \label{item:alpha:unique} there is only one maximal element in $\Ap(S; \{n_1, \ldots, n_r\})$ and all the elements in $\Ap(S; \{n_1, \ldots, n_r\})$ have unique expressions;
      \item \label{item:alpha:ap} $\#\Ap(S; \{n_1, \ldots, n_r\}) = \prod_{i = r+1}^{r+m} \alpha_{i}$.
  \end{enumerate}
\end{proposition}
\begin{proof}~
  \begin{enumerate}[leftmargin=10pt]
  \item[\ref{item:alpha}\kern-3pt] implies \ref{item:alpha:maximal}. It is clear that $\Ap(S; \{n_1, \ldots, n_r\})$ only has one maximal element, which is $\omega = \sum_{i = r+1}^m \alpha_i n_i$. Let $x = \sum_{i = 1}^{r+m} \lambda_i \mathbf{e}_i \in \mathrm{Z}(\omega)$. Note that $\lambda_i = 0$ for every $i \in \{1, \ldots, r\}$. Moreover, since $\omega \in \Ap(S; \{n_1, \ldots, n_r\})$, we find that $\lambda_i n_i \in \Ap(S; \{n_1, \ldots, n_r\})$ and $\lambda_i \le \alpha_i$ for every $i \in \{r+1, \ldots, r+m\}$. Therefore, we have $\lambda_i = \alpha_i$ for every $i \in \{r+1, \ldots, r+m\}$.
  
  \item[\ref{item:alpha:maximal}\kern-3pt] implies \ref{item:alpha:unique}. Note that if $s' \le_S s$ and $s$ has a unique expression, then $s'$ also has a unique expression.
  
  \item[\ref{item:alpha:unique}\kern-3pt] implies \ref{item:alpha:ap}. Let $\omega = \sum_{r+1}^m \lambda_i n_i$ be the maximum of $\Ap(S; \{n_1, \ldots, n_r\})$. For each $i \in \{r+1, \ldots, r+m\}$, we show that $\lambda_i = \alpha_i$. Since $\lambda_i n_i \le_S \omega$, $\lambda_i n_i$ is in $\Ap(S; \{n_1, \ldots, n_r\})$. Moreover, $(\lambda_i+1)n_i$ is not smaller or equal than $\omega$ because $\omega$ has a unique expression. Consequently, we find that $(\lambda_i+1)n_i$ is not in $\Ap(S; \{n_1, \ldots, n_r\})$ and $\lambda_i = \alpha_i$. Finally, we note that if $0 \le \lambda_i \le \alpha_i$ for every $i \in \{r+1, \ldots, r+m\}$, then we have $s = \sum_{i = r+1}^m \lambda_i n_i \in \Ap(S; \{n_1, \ldots, n_r\})$ and, thus, $s$ has a unique expression. Therefore, there are at least $\prod_{i = r+1}^{r+m} \alpha_{i}$ elements in $\Ap(S; \{n_1, \ldots, n_r\})$. The proof is completed by invoking Lemma \ref{lem:alpha:apery}.

  \item[\ref{item:alpha:ap}\kern-3pt] implies \ref{item:alpha}. We find that the inclusion in Lemma \ref{lem:alpha:apery} must be an equality by taking cardinalities. \qedhere
  \end{enumerate}
\end{proof}

Interestingly, simplicial affine semigroups verifying the condition \ref{item:alpha:unique} of the previous proposition had already been studied in the specialized literature, see \cite[Section 3]{free-affine}, where the authors proved that if these semigroups are Cohen-Macaulay, then they are free (\cite[Theorem 3.3]{free-affine}). To prove this result they show that the constants $\alpha_i+1$ and $c_i^*$ coincide for a certain arrangement of the minimal generators of the semigroup. Then, they apply Lemma \ref{lem:free} in conjunction with the Cohen-Macaulay hypothesis. Therefore, their result can be stated as follows.

\begin{theorem}[{\cite[Theorem 3.3]{free-affine}}] \label{thm:alpha:free}
  Let $S$ be a simplical semigroup of $\mathbb{N}^r$ such that it is $\alpha$-rectangular for $\{n_1, \ldots, n_{r}\}$. Then, there is an arrangement $(n_1, \ldots,n_r,n_{r+1},\ldots, n_{r+m})$ of its minimal generators such that $\alpha_{i}+1 = c_{i}^*$ for every $i \in \{r+1, \ldots, r+m\}$. Moreover, if $S$ is Cohen-Macaulay, then $S$ is free for this arrangement. 
\end{theorem}

Recall that numerical semigroups are Cohen-Macaulay. Therefore, the previous result yields that if $S$ is an $\alpha$-rectangular numerical semigroup for $n$, then it is free for an arrangement of its minimal generators starting by $n$. This result was proven in \cite{ci:classes:2} when $n = \mathrm{m}(S)$.

Now we study the set of isolated factorizations of $\alpha$-rectangular semigroups. We make use of the following lemma.

\begin{lemma} \label{lem:iso:simplicial}
  Let $S$ be a simplicial affine semigroup of $\mathbb{N}^r$ minimally generated by $\{n_1, \ldots, n_{r+m}\}$. The following statements are equivalent:
  \begin{enumerate}
      \item \label{item:simplicial:ib} $\mathrm{I}_b(S) \cap \langle \mathbf{e}_{r+1}, \ldots, \mathbf{e}_{r+m} \rangle = \{c_{r+1} \mathbf{e}_{r+1}, \ldots, c_{r+m} \mathbf{e}_{r+m}\}$;
      \item \label{item:simplicial:ap} $\{\sum_{r+1}^{r+m} \lambda_i n_i : 0 \le \lambda_i < c_i\} \subset \{s \in S : \mathfrak{d}(s) = 1\}$.
  \end{enumerate}
  If any of these statements holds, then $\{\sum_{r+1}^{r+m} \lambda_i n_i : 0 \le \lambda_i < c_i\} \subseteq \Ap(S; \{n_1, \ldots, n_r\})$.
\end{lemma}
\begin{proof}~
 \begin{enumerate}[leftmargin=10pt]
      \item[\ref{item:simplicial:ib}\kern-3pt] implies \ref{item:simplicial:ap}. Let $\lambda_{r+1},\ldots,\lambda_{r+m}$ be non negative integers such that $\lambda_i < c_i$ for every $i \in \{r+1, \ldots, r+m\}$. We show that $y = \lambda_{r+1} \mathbf{e}_{r+1} + \cdots + \lambda_{r+m} \mathbf{e}_{r+m} \in \mathrm{I}_s(S)$. If this is not the case, then by Lemma \ref{lem:isolated:2}, there would be $z\in \mathrm{I}_b(S)$ with $z\le y$. Note that $z\in \langle \mathbf{e}_{r+1}, \ldots, \mathbf{e}_{r+m} \rangle$. Thus, by hypothesis, $z=c_i \mathbf{e}_i$ for some $i\in \{r+1,\ldots,r+m\}$, contradicting that $z\le y$.
      \item[\ref{item:simplicial:ap}\kern-3pt] implies \ref{item:simplicial:ib}. First, recall that $\{c_{r+1} \mathbf{e}_{r+1}, \ldots, c_{r+m} \mathbf{e}_{r+m}\} \subseteq \mathrm{I}_b(S)$ (see Theorem \ref{thm:isolated}). Now let $z \in \mathrm{I}_b(S) \cap \langle \mathbf{e}_{r+1}, \ldots, \mathbf{e}_{r+m} \rangle$. From Lemma \ref{lem:isolated} it follows that $z$ is not smaller than any factorization $\lambda_{r+1} \mathbf{e}_{r+1} + \cdots + \lambda_{r+m} \mathbf{e}_{r+m}$ such that $0 \le \lambda_i < c_i$ for every $i \in \{r+1, \ldots, r+m\}$. Therefore, there is $i \in \{r+1, \ldots, r+m\}$ such that $z \ge c_i \mathbf{e}_i$ and, hence, we obtain $z = c_i \mathbf{e}_i$ by the minimality of $z$ (Lemma \ref{lem:isolated:2}). 
  \end{enumerate} 
  Finally, let us assume that \ref{item:simplicial:ap} holds. Then, the elements of $\{\sum_{r+1}^{r+m} \lambda_i n_i : 0 \le \lambda_i < c_i\}$ do not have expressions involving $n_1, \ldots, n_r$. That is, these elements are in $\Ap(S; \{n_1, \ldots, n_r\})$.
\end{proof}

We wonder when a $c$-rectangular semigroup is $\alpha$-rectangular. The following result answers this question. Let us highlight that, as part of the proof of \cite[Theorem 3.3]{free-affine}, the authors show that \ref{item:thm:alpha:alpha} implies \ref{item:thm:alpha:c-ci} in Theorem \ref{thm:alpha:c}. Here we give our own proof relying on Corollary \ref{cor:rectangular}. 

\begin{theorem} \label{thm:alpha:c}
  Let $S$ be a simplicial affine semigroup of $\mathbb{N}^r$ minimally generated by $\{n_1, \ldots, n_{r+m}\}$. The following statements are equivalent:
  \begin{enumerate}
    \item \label{item:thm:alpha:alpha} $S$ is $\alpha$-rectangular for $\{n_1, \ldots, n_{r}\}$;
    \item \label{item:thm:alpha:c-ci} $S$ is $c$-rectangular for $\{n_1, \ldots, n_r\}$ and $c_i n_i \not \in \Ap(S; \{n_1, \ldots, n_r\})$ for every $i \in \{r+1, \ldots, r+m\}$;
    \item \label{item:thm:alpha:c-i} $S$ is $c$-rectangular for $\{n_1, \ldots, n_r\}$ and $\mathrm{I}_b(S) \cap \langle \mathbf{e}_{r+1}, \ldots, \mathbf{e}_{r+m} \rangle = \{c_{r+1} \mathbf{e}_{r+1}, \ldots, c_{r+m} \mathbf{e}_{r+m}\}$; \item \label{item:thm:alpha:unique} $\mathrm{I}_b(S) \cap \langle \mathbf{e}_{r+1}, \ldots, \mathbf{e}_{r+m} \rangle = \{c_{r+1} \mathbf{e}_{r+1}, \ldots, c_{r+m} \mathbf{e}_{r+m}\}$ and the elements in $\Ap(S; \{n_1, \ldots, n_r\})$ are of unique expression;
    \item \label{item:thm:alpha:ci-i} $\mathrm{I}_b(S) \cap \langle \mathbf{e}_{r+1}, \ldots, \mathbf{e}_{r+m} \rangle = \{c_{r+1} \mathbf{e}_{r+1}, \ldots, c_{r+m} \mathbf{e}_{r+m}\}$ and $c_i n_i \not \in \Ap(S; \{n_1, \ldots, n_r\})$ for every $i \in \{r+1, \ldots, r+m\}$;
    \item \label{item:thm:alpha:car} $\mathrm{I}_b(S) \cap \langle \mathbf{e}_{r+1}, \ldots, \mathbf{e}_{r+m} \rangle = \{c_{r+1} \mathbf{e}_{r+1}, \ldots, c_{r+m} \mathbf{e}_{r+m}\}$ and $\# \Ap(S; \{n_1, \ldots, n_r\}) = \prod_{i = r+1}^{r+m} c_i$.
\end{enumerate}
   If any of the statements hold, then $c_i = \alpha_i + 1$ for every $i \in \{r+1, \ldots, r+m\}$.
\end{theorem}
\begin{proof}~
\begin{enumerate}[leftmargin=10pt]
   \item[\ref{item:thm:alpha:alpha}\kern-3pt] implies \ref{item:thm:alpha:c-ci}. Since each of the elements of $\Ap(S; \{n_1, \ldots, n_r\})$ has a unique expression (Proposition \ref{prop:alpha}), we find that $c_j n_j$ is not in $\Ap(S; \{n_1, \ldots, n_r\})$ and, thus, $c_j -1 \ge \alpha_j$. Corollary \ref{cor:rectangular} yields the inequality $c_j -1 \le \alpha_j$.
  \item[\ref{item:thm:alpha:c-ci}\kern-3pt] implies \ref{item:thm:alpha:alpha}. Note that $\alpha_i = c_i-1$ for every $i \in \{r+1, \ldots, r+m\}$ due to the definition of $\alpha_i$. 
  \item[\ref{item:thm:alpha:alpha}\kern-3pt] implies \ref{item:thm:alpha:c-i}. In light of Proposition \ref{prop:alpha}, condition \ref{item:simplicial:ap} of Lemma \ref{lem:iso:simplicial} holds.
  \item[\ref{item:thm:alpha:c-i}\kern-3pt] implies \ref{item:thm:alpha:alpha}. In view of Lemma \ref{lem:iso:simplicial}, we find that the elements of $\Ap(S; \{n_1, \ldots, n_r\})$ have unique expressions and, thus, $S$ is $\alpha$-rectangular for $\{n_1, \ldots, n_r\}$ (Proposition \ref{prop:alpha}).
  \item[\ref{item:thm:alpha:alpha}\kern-3pt] implies \ref{item:thm:alpha:unique}. It follows from the previous implications.
    \item[\ref{item:thm:alpha:unique}\kern-3pt] implies \ref{item:thm:alpha:ci-i}. Recall that $c_i n_i$ has more than one factorization and, thus, $c_i n_i \not \in \Ap(S; \{n_1, \ldots, n_r\})$.
  \item[\ref{item:thm:alpha:ci-i}\kern-3pt] implies \ref{item:thm:alpha:car}.  Lemma \ref{lem:iso:simplicial} shows that
  \[\left\{\sum\nolimits_{r+1}^{r+m} \lambda_i n_i : 0 \le \lambda_i < c_i\right\} \subseteq \Ap(S; \{n_1, \ldots, n_r\}).\]
  Conversely, if $\lambda_{r+1} n_{r+1} + \cdots + \lambda_{r+m}n_{r+m} \in \Ap(S; \{n_1, \ldots, n_r\})$, then, since $c_i n_i \not \in \Ap(S; \{n_1, \ldots, n_r\})$, we have $\lambda_i < c_i$ for every $i \in \{r+1, \ldots, r+m\}$. We have shown that $S$ is $c$-rectangular. From Lemma \ref{lem:iso:simplicial} it follows that the elements of $\Ap(S; \{n_1, \ldots, n_r\})$ have unique expression. We obtain \ref{item:thm:alpha:car} by taking cardinalities in $\Ap(S; \{n_1, \ldots, n_r\})$.
  \item[\ref{item:thm:alpha:car}\kern-3pt] implies \ref{item:thm:alpha:alpha}. From Lemma \ref{lem:iso:simplicial} it follows that
    \[ \left\{\sum\nolimits_{i = r+1}^{r+m} \lambda_i n_i : 0 \le \lambda_i < c_i\right\} \subseteq \Ap(S; \{n_1, \ldots, n_r\}), \]
  where the elements of the first set have unique expressions. By taking cardinalities we find that this inclusion must be an equality and, thus, $S$ is $\alpha$-rectangular for $\{n_1, \ldots, n_r\}$ (Proposition \ref{prop:alpha}). \qedhere  
  \end{enumerate}
\end{proof}

In the case of numerical semigroups, $\# \Ap(S; n_1) = \prod_{i = 2}^{e} c_i$ is equivalent to $n_1 = \prod_{i = 2}^{e} c_i$. 

\begin{corollary} \label{cor:alpha:betti}
  Let $S$ be a simplical semigroup of $\mathbb{N}^r$ minimally generated by $\{n_1, \ldots, n_{r+m}\}$. If $S$ is Cohen-Macaulay and $\alpha$-rectangular for $\{n_1, \ldots, n_{r}\}$, then $\mathrm{Betti}(S) = \mathrm{IBetti}(S) = \{c_{r+1} n_{r+1}, \ldots, c_{r+m} n_{r+m}\}$.
\end{corollary}
\begin{proof}
  In view of Theorem \ref{thm:alpha:free}, $S$ is free for some rearrangement $(n_1, \ldots, n_r, n_{\sigma(r+1)}, \ldots, n_{\sigma(r+m)})$ of its minimal generators, which we may assume to be $(n_1, \ldots, n_{r+m})$ without loss of generality. Moreover, in conjuntion with Theorem \ref{thm:alpha:c}, we have $c_i = \alpha_{i}+1 = c_{i}^*$ for every $i \in \{r+1, \ldots, r+m\}$. Therefore, from Lemma \ref{lem:free} it follows that $\mathrm{Betti}(S) = \{c^*_{r+1} n_{r+1}, \ldots, c^*_{r+m} n_{r+m}\} = \{c_{r+1} n_{r+1}, \ldots, c_{r+m} n_{r+m}\}$, which is contained in $\mathrm{IBetti}(S)$.
\end{proof}

\begin{proposition} \label{prop:alpha:gluing}
  Let $S$ be a Cohen-Macaulay simplicial affine semigroup of $\mathbb{N}^r$. Let $n_1, \ldots, n_r$ be $r$ minimal generators of $S$ such that $\mathrm{L}_{\mathbb{Q}^+}(S) = \mathrm{L}_{\mathbb{Q}^+}(\{n_1, \ldots, n_r\})$. The following statements are equivalent:
  \begin{enumerate}
      \item \label{item:alpha:gluing:1} $S$ is $\alpha$-rectangular for $\{n_1, \ldots, n_r\}$;
      \item \label{item:alpha:gluing:2} $S$ is the gluing of a simplicial affine semigroup $S'$ and $\langle d \rangle \subset \mathbb{N}^r$, where $S'$ is $\alpha$-rectangular for $\{n_1, \ldots, n_r\}$ and $kd \not \in \Ap(S'; \{n_1, \ldots, n_r\})$ for every $k \in \mathbb{Z}^+$.
  \end{enumerate}
\end{proposition}
\begin{proof}~
\begin{enumerate}[leftmargin=10pt]
   \item[\ref{item:alpha:gluing:1}\kern-3pt] implies \ref{item:alpha:gluing:2}. In view of Theorem \ref{thm:alpha:free}, $S$ is free for an arrangement $(n_1, \ldots, n_{r+m})$ of the minimal generators and $\alpha_i + 1= c_i^*$ for every $i \in \{r+1, \ldots, r+m\}$. Therefore, $S$ is the gluing of the free semigroup $S' = \langle n_1, \ldots, n_{r+m-1} \rangle$ and $\langle n_{r+m} \rangle$ (after rearranging the last $m$ generators if needed). In view of Lemma \ref{lem:free}, we find that
  \begin{equation} \label{eq:alpha:gluing}
  \Ap(S'; \{n_1, \ldots, n_r\}) = \left\{\sum\nolimits_{i = r+1}^{r+m-1} \lambda_i n_i : \lambda_i < c_i^*\right\} \subset \left\{\sum\nolimits_{i = r+1}^{r+m} \lambda_i n_i, \lambda_i < c_i^*\right\} = \Ap(S; \{n_1, \ldots, n_r\}).
  \end{equation}
  Since $(c_{r+1}^*-1) n_{r+1} + \cdots + (c_{r+m}^*-1) n_{r+m}$ has a unique expression in $S$, it also has a unique expression in $S'$. Hence, Proposition \ref{prop:alpha} yields that $S'$ is $\alpha$-rectangular for $\{n_1, \ldots, n_r\}$. 
   Let $k \in \mathbb{Z}^+$. There are two cases:
   \begin{itemize}
       \item If $k \ge c_{r+m}^* = \alpha_{r+m}+1$, then $k n_{r+m} \not \in \Ap(S; \{n_1, \ldots, n_r\})$ and, in particular, we find that $k n_{r+m} \not \in \Ap(S'; \{n_1, \ldots, n_r\})$. 
       \item If $k \le \alpha_{r+m}$, then the element $k n_{r+m}$ has a unique expression in $S$ and, thus, $k n_{r+m} \not \in S'$.
   \end{itemize}
   In any case, we have shown that $k n_{r+m} \not \in \Ap(S';\{n_1, \ldots, n_r\})$.
   
  \item[\ref{item:alpha:gluing:2}\kern-3pt] implies \ref{item:alpha:gluing:1}. From Theorem \ref{thm:alpha:free} it follows that $S'$ is free for an arrangement $(n_1, \ldots, n_{r+m-1})$ of its minimal generators. Let $n_{r+m} = d$. Then, since $S$ is the gluing of $S'$ and $\langle n_{r+m} \rangle$, we find that $\{n_1, \ldots, n_{r+m}\}$ is the minimal system generators of $S$. In light of Lemma \ref{lem:free}, $S$ is free and the expression \eqref{eq:alpha:gluing} holds again. Let $i \in \{r+1, \ldots, r+m\}$. We show that $c_i^* = \alpha_i + 1$. Since $(c_i^* -1) n_i \in \Ap(S; \{n_1, \ldots, n_r\})$, we note that $c_i^* \le \alpha_i + 1$. There are two cases:
  
  \begin{itemize}
      \item $i < r+m$. From Theorem \ref{thm:alpha:free} it follows that $c_i^* n_i \not \in \Ap(S'; \{n_1, \ldots, n_r\})$. Note that $c_i^* n_i \ne k n_{r+m}$ for every $k < c_{r+m}^*$. Thus, in view of \eqref{eq:alpha:gluing}, we obtain $c_i^* n_i \not \in \Ap(S; \{n_1, \ldots, n_r\})$ and $c_i^* = \alpha_i + 1$.
      \item $i = r+m$. Since $c^*_{r+m} n_{r+m} \in S'$ but, by hypothesis, $c^*_{r+m} n_{r+m} \not \in \Ap(S'; \{n_1, \ldots, n_r\})$, we can write $c_{r+m}^* n_{r+m} = \lambda_1 n_1 +  \cdots + \lambda_{r+m-1} n_{r+m-1}$, where $\lambda_1, \ldots, \lambda_{r+m-1}$ are non negative integers such that $\lambda_i \ne 0$ for some $i \in \{1, \ldots, r\}$. In particular, we find that $c_{r+m}^* n_{r+m} \not \in \Ap(S; \{n_1, \ldots, n_r\})$ and $c^*_{r+m} = \alpha_{r+m}+1$. \qedhere
  \end{itemize}
  
  \end{enumerate}
\end{proof}

The previous result can be translated to numerical semigroups, generalizing \cite[Theorem 3.3]{ci:classes:2}.

\begin{corollary} \label{cor:alpha:gluing}
  Let $S$ be a numerical semigroup. Then, $S$ is $\alpha$-rectangular for a minimal generator $n_i$ if and only if $S = a S' + b \mathrm{N}$ is the gluing of $S'$ and $\mathbb{N}$, where $S'$ is $\alpha$-rectangular for $n_i / a$ and $b \not \in \Ap(S'; n_i/a)$. 
\end{corollary}

\section{Complete intersection semigroups with only one Betti-minimal element} \label{sec:ci}

In this section we study simplicial affine semigroups that are complete intersections with a single Betti-minimal element. Thanks to our previous results, we are able to characterize these semigroups using isolated factorizations. The graphs associated to each Betti element have the following form. The one corresponding to the only Betti-minimal element has all its factorizations isolated (something that we already knew from Proposition~\ref{prop:all-i}), and the rest will have a connected component corresponding to the non isolated factorizations predicted in Corollary~\ref{cor:isolated}, and the rest of connected components will be isolated factorizations.

\begin{theorem} \label{thm:ci-b1}
  Let $S$ be a simplicial affine semigroup of $\mathbb{N}^r$ with codimension $m$. The following statements are equivalent:
  \begin{enumerate}
    \item \label{item:ci-b1} $S$ is a complete intersection with a single Betti-minimal element;
    \item \label{item:ci-b1:nc} $\mathrm{i}_b(S) = m+1$ and $\mathrm{nc}(\nabla_b) = \mathrm{i}(b) + 1$ for every $b \in \mathrm{Betti}(S) \setminus \{b_1\}$ for some $b_1 \in \mathrm{Betti}\text{-}\mathrm{minimals}(S)$.
  \end{enumerate}
\end{theorem}
\begin{proof}~
  \begin{enumerate}[leftmargin=10pt]  
  \item[\ref{item:ci-b1}\kern-3pt] implies \ref{item:ci-b1:nc}. By hypothesis $S$ is a complete intersection, and thus the cardinality of any minimal presentation for $S$ is $m$. We also now that this cardinality is $\sum_{b \in \mathrm{Betti}(S)} (\mathrm{nc}(\nabla_{b}) -1)$ (Section \ref{sec:pre:presentations}). Let $b_1$ be the only Betti-minimal element of $S$. If $b \in \mathrm{Betti}(S)$ and $b \ne b_1$, then $b$ has a non isolated factorization (Corollary \ref{cor:isolated}), that is, $\mathrm{nc}(\nabla_{b}) > \mathrm{i}(b)$. This assertion in conjunction with Lemma \ref{lem:simplicial:bound} yields
  \[m = \sum_{b \in \mathrm{Betti}(S)} (\mathrm{nc}(\nabla_{b}) -1) \ge -1 + \sum_{b \in \mathrm{Betti}(S)} \mathrm{i}(b) =  \mathrm{i}_b(S) - 1 \ge m. \]
  Consequently, we obtain $\mathrm{nc}(\nabla_{b}) = 1 + \mathrm{i}(b)$ for every $b \in \mathrm{Betti}(S) \setminus \{b_1\}$ and $i_b(S) = m+1$. 
  \item[\ref{item:ci-b1:nc}\kern-3pt] implies \ref{item:ci-b1}. The cardinality of any minimal presentation of $S$ is given by
  \[ \sum_{b \in \mathrm{Betti}(S)} (\mathrm{nc}(\nabla_{b}) -1) = -1 + \sum_{b \in \mathrm{Betti}(S)} \mathrm{i}(b) =  \mathrm{i}_b(S) - 1 = m.\]
  By Proposition~\ref{prop:all-i} the only Betti-minimal element is $b_1$.\qedhere
  \end{enumerate}
\end{proof}

If we also assume that the unique Betti-minimal element of the semigroups studied in Theorem \ref{thm:ci-b1} is in $\langle n_1, \ldots, n_r \rangle$, then we can obtain even more information about $S$.

\begin{theorem} \label{thm:alpha-m+1}
  Let $S$ be a simplicial affine semigroup of $\mathbb{N}^r$ minimally generated by $\{n_1, \ldots, n_{r+m}\}$. The following statements are equivalent:
  \begin{enumerate}
    \item \label{item:alpha-m+1:free} $S$ is free for some arrangement $(n_1, \ldots, n_r, n_{\sigma(r+1)}, \ldots, n_{\sigma(r+m)})$ of its minimal generators, it only has one Betti-minimal element, and  this Betti-minimal element is not in $\Ap(S; \{n_1, \ldots, n_r \})$;
    \item \label{item:alpha-m+1:ci} $S$ is a complete intersection with only one Betti-minimal element, and  this Betti-minimal element is not in $\Ap(S; \{n_1, \ldots, n_r \})$;
    \item \label{item:alpha-m+1:i-m} $\mathrm{i}_b(S) = m+1$, $S$ only has one Betti-minimal element, and this Betti-minimal element is not in the set $\Ap(S; \{n_1, \ldots, n_r \})$;
    \item \label{item:alpha-m+1:alpha} $\mathrm{i}_b(S) = m+1$ and $S$ is $\alpha$-rectangular for $\{n_1, \ldots, n_r\}$.
  \end{enumerate}
  If any of these statements holds, then $S$ is Cohen-Macaulay, $\min \mathrm{Betti}(S) \in \langle n_1, \ldots, n_r \rangle$ and 
  \[\mathrm{Betti}(S) = \mathrm{IBetti}(S) = \{c_{r+1} n_{r+1}, \ldots, c_{r+m} n_{r+m}\}.\]
\end{theorem}
\begin{proof}~
  \begin{enumerate}[leftmargin=10pt]  
  \item[\ref{item:alpha-m+1:free}\kern-3pt] implies \ref{item:alpha-m+1:ci}. Recall that free simplicial affine semigroups are complete intersections.
  \item[\ref{item:alpha-m+1:ci}\kern-3pt] implies \ref{item:alpha-m+1:i-m}. It follows from Theorem \ref{thm:ci-b1}. 
  \item[\ref{item:alpha-m+1:i-m}\kern-3pt] implies \ref{item:alpha-m+1:alpha} In view of Lemma \ref{lem:simplicial:bound}, we have $\mathrm{I}_b(S) = \{x\} \cup \{c_{r+1} \mathbf{e}_{r+1}, \ldots, c_{r+m} \mathbf{e}_{r+m}\}$, where $x \in \langle \mathbf{e}_1, \ldots, \mathbf{e}_r \rangle$. Let $b_1$ be the unique Betti-minimal element of $S$. Since all the factorizations of $b_1$ are isolated (Proposition \ref{prop:all-i}) and $b_1 \not\in \Ap(S; \{n_1, \ldots, n_r \})$, we have $b_1 = \varphi(x)$. Hence, we obtain 
  \[ \Ap(S; \{n_1, \ldots, n_r\}) \subseteq \Ap(S; b_1) = \{s \in S : \mathfrak{d}(S) = 1\}, \]
  where we applied Corollary \ref{cor:ap-b1}. 
  The proof is completed by invoking Theorem \ref{thm:alpha:c}.  
  \item[\ref{item:alpha-m+1:alpha}\kern-3pt] implies \ref{item:alpha-m+1:free}. Again, we have $\mathrm{I}_b(S) = \{x\} \cup \{c_{r+1} \mathbf{e}_{r+1}, \ldots, c_{r+m} \mathbf{e}_{r+m}\}$, where $x \in \langle \mathbf{e}_1, \ldots, \mathbf{e}_r \rangle$. Let $b_1 = \varphi(x)$. Let $b_2$ be a Betti-minimal element. Since the factorizations of $b_2$ are isolated, there is $i \in \{r+1, \ldots, r+m\}$ such that $b_2 = c_i n_i$. In view of Theorem \ref{thm:alpha:c}, we obtain $c_{i}n_{i} \not \in \Ap(S; \{n_1, \ldots, n_r\})$. Therefore, we can write $c_i n_i = \lambda_1 n_1 + \cdots + \lambda_{r+m} n_{r+m}$ for some non negative integers $\lambda_1, \ldots, \lambda_{r+m}$, where $\lambda_j \ne 0$ for some $j \in \{1,\ldots,r\}$. Since $y = \lambda_1 \mathbf{e}_1 + \cdots + \lambda_{r+m} \mathbf{e}_{r+m} \in \mathrm{I}_b(S)$, the only possibility is $y = x$ and $b_2 = b_1$. That is, $S$ only has one Betti minimal element, which is not in $\Ap(S; \{n_1, \ldots, n_r \})$.

  Now we show that $S$ is Cohen-Macaulay with the help of Proposition \ref{prop:cm}. Let $s \in S$.  There are unique $\alpha \in \Ap(S; b_1)$ and $q \in \mathbb{N}$ such that $s = \alpha + q b_1$ (Section \ref{sec:pre:apery}). There are $\omega \in \Ap(S; \{n_1, \ldots, n_r\})$ and $a \in \langle n_1, \ldots, n_r \rangle$ such that $\alpha = \omega + a$. We have $s = \omega + (a + qb_1)$. We show that this is the only way to write $s$ as a sum of two elements of $\Ap(S; \{n_1, \ldots, n_r\})$ and $\langle n_1, \ldots, n_r \rangle$. Let $s = \omega' + n$ with $\omega' \in \Ap(S; \{n_1, \ldots, n_r\})$ and $n \in \langle n_1, \ldots, n_r \rangle$. Let $\mathrm{Z}(\omega') = \{w'\}$ (this is a singleton in light of Proposition \ref{prop:alpha}) and $n = \lambda_1 n_1 + \cdots + \lambda_r n_r$, where $\lambda_1, \ldots, \lambda_r$ are non negative integers. Note that the only element of $\mathrm{I}_b(S)$ that may be smaller or equal than $y = w'+\lambda_1 \mathbf{e}_1 + \cdots + \lambda_r \mathbf{e}_r$ is $x$ (if $c_i\mathbf{e}_i<y$, for some $i\in\{r+1,\ldots,r+m\}$, then $c_i\mathbf e_i<w'$, but $\omega'\in \Ap(S; \{n_1, \ldots, n_r\})$ and, by Theorem \ref{thm:alpha:c}, $c_in_i\not\in \Ap(S; \{n_1, \ldots, n_r\})$). Let us substract $x$ from $y$ as many times as possible to obtain $y = w' +x' + q'x$, where $x$ is not smaller than $x'$. In view of Lemma \ref{lem:isolated:2}, we have $w'+x' \in \mathrm{I}_s(S)$ and, thus, $\omega' + \varphi(x') \in \Ap(S; b_1)$. We find that $q = q'$ and $\omega' + \varphi(x') = \omega+a$. Since $\omega+a$ has a unique expression due to Corollary \ref{cor:ap-b1}, we obtain $\omega = \omega'$. 
  
  Finally, Theorem \ref{thm:alpha:free} asserts that $S$ is free for such an arrangement of its minimal generators.
  \end{enumerate}
  The elements of $\mathrm{Betti}(S)$ are determined by Corollary \ref{cor:alpha:betti}.
\end{proof}

Let $S$ be a numerical semigroup minimally generated by $\{n_1, \ldots, n_e\}$. Let us assume that $\mathrm{i}_b(S) = e$ and $\mathrm{Betti}\text{-}\mathrm{minimals}(S) = \{b_1\}$. Since $c_i\mathbf e_i$ is an isolated factorization for all possible $i$, $\mathrm{I}_b(S) = \{c_1\mathbf{e}_1, \ldots, c_e\mathbf{e}_e\}$. We can rearrange the minimal generators of $S$ so that  $c_1 n_1 \le c_i n_i$ for every $i \in \{1, \ldots, e\}$. Betti-minimal elements have only isolated factorizations. Thus, the only possibility is $b_1 = c_1 n_1 = c_2 n_2$. Therefore, $S$ verifies the condition \ref{item:alpha-m+1:i-m} of Theorem \ref{thm:alpha-m+1} for this order of its minimal generators. This claim in conjunction with Theorem \ref{thm:ci-b1} proves the following interesting result, which characterizes those complete intersection numerical semigroups with only one Betti-minimal element.

\begin{corollary} \label{cor:ci-b1}
  Let $S$ be a numerical semigroup. Let $(n_1, \ldots, n_e)$ be an arrangement of the minimal system of generators of $S$ such that $c_1 n_1 \le c_i n_i$ for every $i \in \{1, \ldots, e\}$. The following statements are equivalent:
  \begin{enumerate}
    \item \label{item:ci:free} $S$ is a free numerical semigroup with only one Betti-minimal element;
    \item \label{item:ci:ci} $S$ is a complete intersection numerical semigroup with only one Betti-minimal element;
    \item \label{item:ci:i-b1} $\mathrm{i}_b(S) = e$ and $S$ only has one Betti-minimal element;
    \item \label{item:ci:alpha} $\mathrm{i}_b(S) = e$ and $S$ is $\alpha$-rectangular for $n_1$.
  \end{enumerate}
  If any of these statements holds, then $\mathrm{Betti}(S) = \mathrm{IBetti}(S) = \{c_2 n_2, \ldots, c_e n_e\}$ and $n_1 = \prod_{i = 2}^e c_i$.
\end{corollary}
\begin{proof}
  We only have to show that $n_1 = \prod_{i = 2}^e c_i$. This follows from the fact that $S$ is $\alpha$-rectangular for $n_1$ and, thus, $\# \Ap(S; n_1) = \prod_{i = 2}^e c_i$ (Theorem \ref{thm:alpha:c}).
\end{proof}

\begin{example}
    Let $S=\langle 16,20,30,45\rangle$. Then $\mathrm{Betti}(S)=\{60,80,90\}$. It follows that $S$ has a single Betti-minimal element. Also,
    $\mathrm Z(60)= \{ ( 0, 3, 0, 0 ), ( 0, 0, 2, 0 )\}$, $\mathrm Z(80) = \{ ( 5, 0, 0, 0 ), ( 0, 4, 0, 0 ), ( 0, 1, 2, 0 )\}$, and $\mathrm Z(90) = \{ ( 0, 3, 1, 0 ), ( 0, 0, 3, 0 ), ( 0, 0, 0, 2 ) \}$, whence \[\mathrm I_b(S)=\{ ( 0, 3, 0, 0 ), ( 0, 0, 2, 0 ), (5,0,0,0), ( 0, 0, 0, 2 )\}.\]
    In view of Corollary \ref{cor:ci-b1}, $S$ is $\alpha$-rectangular for $n_1=20$.
\end{example}

The semigroups studied in Theorem \ref{thm:alpha-m+1} are recursively characterized in terms of gluings.

\begin{corollary} \label{cor:ci-b1:gluing}
 Let $S$ be a simplicial affine semigroup of $\mathbb{N}^r$ minimally generated by $\{n_1, \ldots, n_{r+m}\}$. The following statements are equivalent:
  \begin{enumerate}
      \item \label{item:ci1:gluing:a} $S$ is a complete intersection with $\mathrm{Betti}\text{-}\mathrm{minimals}(S)$ a singleton not intersecting $\Ap(S; \{n_1, \ldots, n_r \})$;
      \item \label{item:ci1:gluing:b} $S$ is the gluing of a simplicial affine semigroup $S'$ of $\mathbb{N}^r$ and $\langle d \rangle \subset \mathbb{N}^r$, where either the codimension of $S'$ is $0$ or $S'$ is a complete intersection with only one minimal Betti element, say $b_1$, such that $b_1 \not\in \Ap(S; \{n_1, \ldots, n_r \})$ and $b_1 \le_{S'} c_d d$, with $c_d = \min \{k \in \mathbb{Z}^+ : k d \in S'\}$. 
  \end{enumerate}
\end{corollary}
\begin{proof}~
\begin{enumerate}[leftmargin=10pt]
   \item[\ref{item:ci1:gluing:a}\kern-3pt] implies \ref{item:ci1:gluing:b}. From Theorem \ref{thm:alpha-m+1} it follows that $S$ is Cohen-Macaulay, $\alpha$-rectangular for $\{n_1, \ldots, n_r\}$ and free for an arrangement $(n_1, \ldots, n_r, n_{\sigma(r+1)}, \ldots, n_{\sigma(r+m)})$ of its minimal generators. Moreover, we have
   \[\mathrm{Betti}(S) = \{c_{r+1} n_{r+1}, \ldots, c_{r+m} n_{r+m}\} \quad \text{and} \quad \mathrm{I}_b(S) = \{x, c_{r+1} \mathbf{e}_{r+1}, \ldots, c_{r+m} \mathbf{e}_{r+m}\},\]
   where $x \in \langle \mathbf{e}_1, \ldots, \mathbf{e}_r \rangle$. Let us assume that this arrangement is $(n_1, \ldots, n_{r+m})$ without loss of generality. Hence, we have $c_i = \alpha_i+1 = c_i^* = \bar{c}_i$ for every $i \in \{r+1, \ldots, r+m\}$ (Theorems \ref{thm:alpha:free} and \ref{thm:alpha:c}). From the proof of Theorem \ref{thm:alpha-m+1}, it follows that $b_1=\varphi(x)$ is the only Betti-minimal element.  Since all the factorizations of $b_1$ are isolated (Proposition \ref{prop:all-i}), we have $b_1 = \varphi(x) = c_i n_i$ for some $i \in \{r+1, \ldots, r+m\}$. Note that the arrangement
   \[ (n_1, \ldots, n_r, n_i, n_{r+1}, \ldots, n_{i-1}, n_{i+1}, \ldots, n_{r+m})\]
   verifies the statement b) of Lemma \ref{lem:free} and, thus, $S$ is free for this arrangement. Thus, we can also assume that $i = r+1$. Recall that $S$ is the gluing of the free semigroup $S' = \langle n_1, \ldots, n_{r+m-1}\rangle$ and $\langle n_{r+m} \rangle$. Therefore, by \eqref{eq:betti-gen}, $\mathrm{Betti}(S) = \mathrm{Betti}(S') \cup \{c_{r+m} n_{r+m}\}$. We note that $x, c_{r+1} \mathbf{e}_{r+1} \in \mathrm{Z}_{S'}(b_1) \subseteq \mathrm{Z}_S(b_1)$ and, thus, $b_1 \in \mathrm{Betti}(S')$. As a consequence, $\mathrm{Betti}\text{-}\mathrm{minimals}(S') = \{b_1\}$. Finally, we can write $c_{r+m} n_{r+m} = c_{r+1} n_{r+1} + \lambda_{1} n_1 + \cdots + \lambda_{r+m} c_{r+m}$ for some non negative integers $\lambda_1, \ldots, \lambda_{r+m}$. Note that $\lambda_{r+m}$ must be $0$ due to the definition of $c_{r+m}$. That is, $b_1 = c_{r+1} n_{r+1} \le_{S'} c_{r+m} n_{r+m}$.
   
  \item[\ref{item:ci1:gluing:b}\kern-3pt] implies \ref{item:ci1:gluing:a}. In view of Theorem \ref{thm:alpha-m+1}, $S'$ is free and, thus, $S$ is also free. Moreover, we have $\mathrm{Betti}(S) = \mathrm{Betti}(S') \cup \{c_d d\}$. Since $b_1 \le_{S'} c_d d$, it follows that $\mathrm{Betti}\text{-}\mathrm{minimals}(S) = \{b_1\}$. Finally, recall that $b_1 \in \langle n_1, \ldots, n_r \rangle$ and, thus, $b_1 \not\in \Ap(S; \{n_1, \ldots, n_r \})$. \qedhere
  \end{enumerate}
\end{proof}

As a consequence, we can construct an infinitive number of complete intersection simplicial affine semigroups of $\mathbb{N}^r$ with only one minimal Betti element. Therefore, the bound given in Lemma \ref{lem:simplicial:bound} is attained for an infinite number of free affine semigroups. Of course, not all free affine semigroups are of this form. For instance, consider the numerical semigroup $S = \langle 4, 6, 5 \rangle$. In Example \ref{ex:disjoint} we showed that $\mathrm{i}_b(S) = 4$. Moreover, note that $S$ is free for the arrangement $(4,6,5)$. Therefore, it does not have a unique Betti-minimal element.

Finally, note that when Corollary \ref{cor:ci-b1:gluing} is stated for numerical semigroups, the hypothesis $b_1 \in \Ap(S; n_1)$ is not needed, provided that $n_1$ is the minimal generator that verifies $c_1 n_1 \le c_i n_i$ for every $i \in \{1, \ldots, e\}$.

\begin{corollary}
  Let $S$ be a numerical semigroup. The following statements are equivalent:
  \begin{enumerate}
      \item \label{item:ci1:gluing:num:a} $S$ is a complete intersection with only one minimal Betti element;
      \item \label{item:ci1:gluing:num:b} $S = a S' + b \mathbb{N}$ is the gluing of a numerical semigroup $S'$ and $\mathbb{N}$, where either $S' = \mathbb{N}$ or $S$ is a complete intersection with only one minimal Betti element such that $\min \mathrm{Betti}(S') \le_{S'} b$. 
  \end{enumerate}
\end{corollary}

\section{Betti sorted semigroups} \label{sec:betti-sorted}

Let $M$ be a monoid with the ascending chain condition on principal ideals. We say that $M$ is \emph{Betti sorted} if its Betti elements are totally ordered with respect to $\le_M$. Note that if it is the case, then $M$ has a finite number of Betti elements. Analogously, we say that $M$ is \emph{Betti-isolated sorted} if the set $\mathrm{IBetti}(M)$ is totally ordered with respect to $\le_M$. It is clear that every Betti sorted monoid is Betti-isolated sorted. 
 Again, we note that if $M$ is Betti-isolated sorted, then $\mathrm{IBetti}(M)$ is finite. Moreover, in light of Lemma \ref{lem:disjoint-betti}, the factorizations in $\mathrm{I}_b(S)$ are disjoint. Therefore, we can write
 \begin{equation} \label{eq:betti-sorted:I}
     \mathrm{I}_b(S) = \Omega \cup \{c \mathbf{e}_a : a \in \mathcal{C}(M)\},
 \end{equation}
 where $\Omega \subset \langle \{\mathbf{e}_a : a \in \mathcal{A}(M) \setminus \mathcal{C}(M) \} \rangle$ consists of disjoint factorizations.
 We will use this fact several times in this section. 

\begin{example}
    Let $S=\langle 4,6,9\rangle$. Then $\mathrm{Betti}(S)=\{12,18\}$, and $12\le_S 18$.
    \[
    \mathrm{Z}(12)=\{(3,0,0),(0,2,0)\},\ 
    \mathrm{Z}(18)=\{(3,1,0),(0,3,0),(0,0,2)\}.
    \]
    Then $\mathrm{I}_b(S)=\{(3,0,0),(0,2,0),(0,0,2)\}=\{c \mathbf e_a: a\in \mathcal C(S)\}$ as \eqref{eq:betti-sorted:I} predicted (for numerical semigroups $\Omega$ is empty).
\end{example}

In the rest of the section we focus on those simplicial affine semigroups that are Betti-isolated sorted. We show that under certain hypothesis these semigroups exhibit several interesting properties. 

\begin{theorem} \label{thm:betti-sorted:alpha}
  Let $S$ be a simplicial affine semigroup of $\mathbb{N}^r$ minimally generated by $\{n_1, \ldots, n_{r+m}\}$. Let us assume that $S$ is Betti-isolated sorted. The following statements are equivalent:
  \begin{enumerate}
      \item \label{item:betti-sorted:min-betti} $\min \mathrm{Betti}(S) \not \in \Ap(S; \{n_1, \ldots, n_r\})$;
      \item \label{item:betti-sorted:alpha} $S$ is $\alpha$-rectangular for $\{n_1, \ldots, n_r\}$;
      \item \label{item:betti-sorted:c} $S$ is $c$-rectangular for $\{n_1, \ldots, n_r\}$;
      \item \label{item:betti-sorted:cini} $c_i n_i \not \in \Ap(S; \{n_1, \ldots, n_r\})$ for every $i \in \{r+m, \ldots, r+m\}$;
      \item \label{item:betti-sorted:car} $\# \Ap(S; \{n_1, \ldots, n_r\}) = \prod_{i = r+1}^{r+m} c_i$.
  \end{enumerate}
\end{theorem}
\begin{proof}
  The fact that \ref{item:betti-sorted:alpha}, \ref{item:betti-sorted:c}, \ref{item:betti-sorted:cini} and \ref{item:betti-sorted:car} are equivalent follows from Theorem \ref{thm:alpha:c} and \eqref{eq:betti-sorted:I}. Let $b_1 = \min \mathrm{Betti}(S)$.
\begin{enumerate}[leftmargin=10pt]
    \item[\ref{item:betti-sorted:min-betti}\kern-3pt] implies \ref{item:betti-sorted:cini}. We show that every element of $\Ap(S; \{n_1, \ldots, n_r\})$ has a unique expression and, thus, $c_i n_i \not \in \Ap(S; \{n_1, \ldots, n_r\})$ for every $i \in \{r+m, \ldots, r+m\}$. From our hypothesis we obtain
    \begin{equation*} \label{eq:betti-sorted:ap}
      \Ap(S; \{n_1, \ldots, n_r\}) \subseteq \Ap(S; b_1) = \{s \in S : \mathfrak{d}(s) = 1\},        
    \end{equation*} 
   where we used Corollary \ref{cor:ap-b1}. 
  \item[\ref{item:betti-sorted:alpha}\kern-3pt] implies \ref{item:betti-sorted:min-betti}. Since the elements of $\Ap(S; \{n_1, \ldots, n_r\})$ have unique expressions (Proposition \ref{prop:alpha}), we note that $b_1 \not \in \Ap(S; \{n_1, \ldots, n_r\})$. \qedhere
\end{enumerate}
\end{proof}

Let $S$ be an Betti-isolated sorted numerical semigroup minimally generated by $\{n_1, \ldots, n_e\}$. Then, equation \eqref{eq:betti-sorted:I} states that $\mathrm{I}_b(S) = \{c_1 \mathbf{e}_1, \ldots, c_e \mathbf{e}_e\}$. Let us assume that $(n_1, \ldots, n_e)$ is an arrangement of the minimal generators of $S$ such that $c_1 n_1 \le_S c_i$ for every $i \in \{1, \ldots, e\}$. We find that $\min \mathrm{Betti}(S) = c_1 n_1$ and, thus, the condition \ref{item:betti-sorted:min-betti} holds for $S$. We have proven the following result.

\begin{corollary}
  Let $S$ be an Betti-isolated sorted numerical semigroup. Let $(n_1, \ldots, n_e)$ be an arrangement of the minimal generators of $S$ such that $c_1 n_1 \le_S c_i$ for every $i \in \{1, \ldots, e\}$. Then $S$ is $\alpha$-rectangular for $n_1$ and, in particular, free for an arrangement $(n_1, n_{\sigma(2)}, \ldots, n_{\sigma(e)})$.
\end{corollary}

We can obtain more information about some Betti-isolated sorted semigroups as a consequence of Theorems \ref{thm:alpha:free} and \ref{thm:betti-sorted:alpha}.

\begin{theorem} \label{thm:betti-sorted}
  Let $S$ be a simplicial affine semigroup of $\mathbb{N}^r$ minimally generated by $\{n_1, \ldots, n_{r+m}\}$. The following statements are equivalent:
  \begin{enumerate}
    \item \label{item:betti-sorted} $S$ is Betti sorted, Cohen-Macaulay and $\min \mathrm{Betti}(S) \not \in  \Ap(S; \{n_1, \ldots, n_r\})$;
    \item \label{item:isolated-betti-sorted} $S$ is Betti-isolated sorted, Cohen-Macaulay and $\min \mathrm{Betti}(S) \not \in  \Ap(S; \{n_1, \ldots, n_r\})$;
    \item \label{item:betti-sorted:presen} $S$ admits a minimal presentation of the form
    \[ \left\{(c_{r+1} \mathbf{e}_{r+1}, \sum\nolimits_{j = 1}^{r}a_j^{r+1} \mathbf{e}_j) \right\} \cup \left\{(c_i \mathbf{e}_i, \sum\nolimits_{j = 1}^{i-1} a_j^i \mathbf{e}_j + c_{i-1} \mathbf{e}_{i-1}) \colon  i \in \{r+2, \ldots, r+m\}\right\}\]
    for a rearrangement of the minimal generators $n_{r+1}, \ldots, n_{r+m}$.
    \item \label{item:betti-sorted:presen:2} $S$ admits a minimal presentation of the form
    \[ \left\{(a_{r+1} \mathbf{e}_{r+1}, \sum\nolimits_{j = 1}^{r}a_j^{r+1} \mathbf{e}_j) \right\} \cup \left\{(a_i \mathbf{e}_i, \sum\nolimits_{j = 1}^{i-1} a_j^i \mathbf{e}_j + a_{i-1} \mathbf{e}_{i-1}) \colon  i \in \{r+2 \ldots, r+m\}\right\}\]
    for a rearrangement of minimal generators $n_{r+1}, \ldots, n_{r+m}$ and some $a_{r+1}, \ldots, a_{r+m} \in \mathbb{N}$.
  \end{enumerate}
  If any of these statements holds, then $\mathrm{IBetti}(S) = \mathrm{Betti}(S) = \{c_{r+1} n_{r+1}, \ldots, c_{r+m} n_{r+m}\}$.
\end{theorem}
\begin{proof}
  It is clear that \ref{item:betti-sorted} implies  \ref{item:isolated-betti-sorted} and \ref{item:betti-sorted:presen} implies \ref{item:betti-sorted:presen:2} hold true.  
  
 \begin{enumerate}[leftmargin=10pt]
 \item[\ref{item:isolated-betti-sorted}\kern-3pt] implies \ref{item:betti-sorted:presen}. Since $S$ is Betti-isolated sorted, there is a rearrangement of the generators $n_{r+1}, \ldots, n_{r+m}$ such that $c_{r+1} n_{r+1} \le_S \cdots \le_S c_{r+m} n_{r+m}$. In light of Theorems \ref{thm:betti-sorted:alpha} and \ref{thm:alpha:free}, $S$ is free. As a consequence, Corollary \ref{cor:alpha:betti} states that the Betti elements of $S$ are  $c_{r+1} n_{r+1}, \ldots, c_{r+m} n_{r+m}$. Moreover, by  Theorem \ref{thm:alpha-m+1}, we have $\mathrm{i}_b(S) = m + 1$ and, hence, $\mathrm{I}_b(S) = \{x_0\} \cup \{c_{r+1}\mathbf{e}_{r+1}, \ldots, c_{r+m}\mathbf{e}_{r+m}\}$ for some $x_0 \in \langle \mathbf{e}_1, \ldots, \mathbf{e}_r \rangle$. Since all the factorizations of $b_1 = \min \mathrm{Betti}(S)$ are isolated and $b_1 \not \in \Ap(S; \{n_1, \ldots, n_r\})$, the only possibility is $\varphi(x_0) = b_1 = c_{r+1} n_{r+1}$. In addition, Theorem \ref{thm:ci-b1} concludes that $\mathrm{nc}(\nabla_{b}) = 1 + \mathrm{i}(b)$ for every $b \in \mathrm{Betti}(S) \setminus \{b_1\}$. All this information allows us to construct a minimal presentation of $S$.  For each $i \in \{r+2, \ldots, r+m\}$ define $z_i$ as $c_{i-1} \mathbf{e}_{i-1}$ if $c_i n_i = c_{i-1} n_{i-1}$. Otherwise, note that $c_{i-1} n_{i-1} <_S c_i n_i$. Therefore, we can choose $z_i \in \mathrm{Z}(c_i n_i)$ with $c_{i-1} \mathbf{e}_{i-1} \le z$. Finally, define $z_{r+1} = x_0$. By Lemma \ref{lem:disjoint-betti} we know that $z_i$ is disjoint with any isolated factorization of a Betti element greater or equal than $c_i n_i$. Hence, we can write $z_i = \sum\nolimits_{j = 1}^{i-1} a_j^i \mathbf{e}_j + c_{i-1} \mathbf{e}_{i-1}$ for every $i \in \{r+2, \ldots, r+m\}$ and $z_{r+1} = \sum\nolimits_{j = 1}^{r} a_j^{r+1} \mathbf{e}_j$. The set
  \[  \left\{(c_{r+1} \mathbf{e}_{r+1}, \sum\nolimits_{j = 1}^{r}a_j^{r+1} \mathbf{e}_j \right\} \cup \left\{(c_i \mathbf{e}_i, \sum\nolimits_{j = 1}^{i-1} a_j^i \mathbf{e}_j + c_{j-1} \mathbf{e}_{j-1}) \colon i \in \{r+2, \ldots, r+m\}\right\}\]
  is a minimal presentation of $S$ due to the characterization introduced in Section \ref{sec:pre:presentations}. 
  
  \item[\ref{item:betti-sorted:presen:2}\kern-3pt] implies \ref{item:betti-sorted}. If $S$ admits a presentation of the form given in \ref{item:betti-sorted:presen:2}, then we have 
  \[ \mathrm{Betti}(S) = \{a_{r+1} n_{r+1}, \ldots, a_{r+m} n_{r+m}\} \quad \text{and} \quad a_{r+1} n_{r+1} \le_S \cdots \le_S a_{r+m} n_{r+m}.\] 
  That is, $S$ is Betti sorted. Furthermore, $S$ has a presentation with cardinality $m$, that is, $S$ is a complete intersection. Finally, we note that $a_{r+1} n_{r+1} = \min \mathrm{Betti}(S)$ has a factorization in $\langle \mathbf{e}_1, \ldots, \mathbf{e}_r \rangle$ and, thus, $\min \mathrm{Betti}(S) \not \in  \Ap(S; \{n_1, \ldots, n_r\})$.   Theorem \ref{thm:alpha-m+1} asserts that $S$ is Cohen-Macaulay. \qedhere
  \end{enumerate}
\end{proof}


In the case of numerical semigroups, we can again get rid of the hypothesis $b_1 \not \in \Ap(S; n_1)$ by choosing an appropiate arrangement of the minimal generators of $S$.

\begin{corollary} \label{cor:betti-sorted}
  Let $S$ be a numerical semigroup. Let $(n_1, \ldots, n_e)$ be an arrangement of the minimal generators of $S$ such that $c_1 n_1 \le c_2 n_2 \le \cdots \le c_e n_e$. The following statements are equivalent:
  \begin{enumerate}
    \item \label{item:betti-sorted:num} $S$ is Betti sorted;
    \item \label{item:isolated-betti-sorted:num} $S$ is Betti-isolated sorted;
    \item \label{item:betti-sorted:presen:num} $S$ admits a minimal presentation of the form
    \[\left\{(c_i \mathbf{e}_i, \sum\nolimits_{j = 1}^{i-1} a_j^i \mathbf{e}_j + c_{i-1} \mathbf{e}_{i-1}) \colon  i \in \{2, \ldots, e\}\right\};\]
    \item \label{item:betti-sorted:presen:2:num} $S$ admits a minimal presentation of the form
    \[\left\{(a_i \mathbf{e}_i, \sum\nolimits_{j = 1}^{i-1} a_j^i \mathbf{e}_j + a_{i-1} \mathbf{e}_{i-1}) \colon  i \in \{2, \ldots, e\}\right\}\]
    for some non negative integers $a_1, \ldots, a_e$.
  \end{enumerate}
  If any of these statements holds, then $\mathrm{IBetti}(S) = \mathrm{Betti}(S) = \{c_{2} n_{2}, \ldots, c_e n_e\}$.
\end{corollary}

Theorem \ref{thm:betti-sorted} allows us to characterize this family of Betti sorted semigroups in terms of gluings.

\begin{corollary} \label{cor:betti-sorted:gluing}
 Let $S$ be a simplicial affine semigroup of $\mathbb{N}^r$ minimally generated by $\{n_1, \ldots, n_{r+m}\}$. The following statements are equivalent:
  \begin{enumerate}
      \item \label{item:betti-sorted:gluing:a} $S$ is Betti sorted, Cohen-Macaulay and $\min \mathrm{Betti}(S) \not \in  \Ap(S; \{n_1, \ldots, n_r\})$;
      \item \label{item:betti-sorted:gluing:b} $S$ is the gluing of a Betti sorted simplicial affine semigroup $S'$ of $\mathbb{N}^r$ and $\langle d \rangle \subset \mathbb{N}^r$, where either the codimension of $S'$ is $0$ or $S'$ is Cohen-Macaulay, $\min \mathrm{Betti}(S') \not \in  \Ap(S; \{n_1, \ldots, n_r\})$ and $\max \mathrm{Betti}(S') \le_{S'} c_d d$ for $c_d = \min \{k \in \mathbb{Z}^+ : k d \in S'\}$. 
  \end{enumerate}
\end{corollary}
\begin{proof}~
\begin{enumerate}[leftmargin=10pt]
   \item[\ref{item:betti-sorted:gluing:a}\kern-3pt] implies \ref{item:betti-sorted:gluing:b}. In view of Theorem \ref{thm:betti-sorted}, there is an arrangement $(n_{r+1}, \ldots, n_{r+m})$ of the latest $m$ minimal generators of $S$ such that $S$ has a presentation of the form 
  \[\{(c_i \mathbf{e}_i, \sum\nolimits_{j = 1}^{i-1} a_j^i \mathbf{e}_j + c_{i-1} \mathbf{e}_{i-1}) \colon i \in \{r+1, \ldots, r+m\}\}.\] 
  Hence, $S$ is free for that arrangement (Lemma \ref{lem:free}).  Recall that $c_e = c_e^*$. Set $S' = \langle n_1, \ldots, n_{r+m-1}\rangle$. Note that $S$ is the gluing of $S'$ and $\langle n_{r+m} \rangle$. Moreover, $\mathrm{Betti}(S) = \{c_{r+1}n_{r+1} \le_{S'} \cdots \le_{S'} c_{r+m} n_{r+m}\}$. A minimal presentation of $S'$ is
  \[\{(c_i \mathbf{e}_i, \sum\nolimits_{j = 1}^{i-1} a_j^i \mathbf{e}_j + c_{i-1} \mathbf{e}_{i-1}) \colon i \in \{r+1, \ldots, r+m-1\}\}.\]
  The result follows from Theorem \ref{thm:betti-sorted}. 
  \item[\ref{item:betti-sorted:gluing:b}\kern-3pt] implies \ref{item:betti-sorted:gluing:a}. In view of Theorem \ref{thm:betti-sorted}, there is an arrangement $(n_{r+1}, \ldots, n_{r+m-1})$ of the last $m-1$ minimal generators of $S'$ such that $S'$ has a presentation of the form 
  \[\rho' = \{(c_i \mathbf{e}_i, \sum\nolimits_{j = 1}^{i-1} a_j^i \mathbf{e}_j + c_{i-1} \mathbf{e}_{i-1}) \colon i \in \{r+1, \ldots, r+m-1\}\}.\]
  Write $c_d d = c_{r+m-1} n_{r+m-1} + s$, where $s \in S'$. There are non negative integers $\lambda_1, \ldots, \lambda_{n+r-1}$ such that $s = \lambda_1 n_1 + \cdots + \lambda_{r+m-1} n_{r+m-1}$. Therefore,
  since $S$ is the gluing of $S'$ and $\langle d \rangle$, the set $\rho = \rho' \cup \{(x, c_d \mathbf{e}_d)\}$, where $x = c_{r+m-1} \mathbf{e}_{r+m-1} + \lambda_1 \mathbf{e}_1 + \cdots + \lambda_{r+m-1} \mathbf{e}_{r+m-1}$, is a minimal presentationofr $S$. The proof is completed by invoking Theorem \ref{thm:betti-sorted}.
  \qedhere
\end{enumerate}
\end{proof}

Again, the result is significantly simplified when numerical semigroups are considerd.

\begin{corollary} \label{cor:betti-sorted:gluing:ns}
  A numerical semigroup $S$ is Betti sorted if and only if there is there is another Betti sorted numerical semigroup $S'$ such that $S = d S' +_{d n} n \mathbb{N}$, where $S' = \mathbb{N}$ or $\max \mathrm{Betti}(S') \le_{S'} n$.
\end{corollary}

\section{Betti divisible semigroups} \label{sec:betti-divisible}

In the sequel we will write $a \mid b$ when $a, b \in \mathbb{Z}^r$ and $b = k a$ for some $k \in \mathbb{Z}$, which is clearly an order relation. Moreover, we will use the notation  $B=\{b_1 \mid \cdots \mid b_k\}$ when $B=\{b_1, \ldots, b_k\}$ and $b_1 \mid b_2 \mid \cdots \mid b_k$.

Let $M$ be a monoid with the ascending chain condition on principal ideals. We say that $M$ is \emph{Betti divisible} if its Betti elements are totally ordered by divisibility, that is, $\mathrm{Betti}(S)$ is of the form $\{b_1 \mid \cdots \mid b_k\}$. 
Betti divisible monoids are Betti sorted.  Analogously, we say that $M$ is \emph{Betti-isolated divisible} if the set $\mathrm{IBetti}(M)$ is totally ordered with respect to the divisibility relation. It is clear that every Betti divisible monoid is Betti-isolated divisible.

\begin{example} \label{ex:betti-divisible:1}
	The numerical semigroup $S = \langle 30, 42, 105, 140 \rangle$ is Betti divisible. We have $\mathrm{Betti}(S) = \{210, 420\}$ and $\mathrm{F}(S) = 523$. Our computations with GAP show that $S$ is the  Betti divisible numerical semigroup with the smallest Frobenius number from those with at least $4$ minimal generators.
\end{example}

The following result
exhibits the strength of the Betti-isolated divisible hypothesis.

\begin{lemma} \label{lem:betti-divisible:cm}
  Let $S$ be a simplicial affine semigroup of $\mathbb{N}^r$ minimally generated by $\{n_1, \ldots, n_{r+m}\}$. If $S$ is Betti-isolated divisible and $b_1 = \min \mathrm{Betti}(S) \not \in \Ap(S; \{n_1, \ldots, n_r\})$, then $S$ is Cohen-Macaulay and 
  \[ \mathrm{I}_b(S) = \{x, c_{r+1} \mathbf{e}_{r+1}, \ldots, c_{r+m} \mathbf{e}_{r+m}\},\] 
  where $x \in \langle \mathbf{e}_1, \ldots, \mathbf{e}_r \rangle \cap \mathrm{Z}(b_1)$.
\end{lemma}
\begin{proof}
  First, we determine the set $\mathrm{I}_b(S)$. Let $x, y \in \mathrm{I}_b(S) \setminus \{c_{r+1} \mathbf{e}_{r+1}, \ldots, c_{r+m} \mathbf{e}_{r+m}\}$. Recall that, by equation \eqref{eq:betti-sorted:I}, we have $x, y \in \langle \mathbf{e}_1, \ldots, \mathbf{e}_r \rangle$. In addition, either $x$ and $y$ are disjoint or $x = y$. Since $\varphi(x)$ and $\varphi(y)$ are isolated Betti elements, one of them is a multiple of the other. Let us assume that $\varphi(y) = k \varphi(x)$ for some $k \in \mathbb{Z}^+$. Then $kx$ is a factorization of $\varphi(y)$. Since $S$ is simplicial, the elements $n_1, \ldots, n_r$ are linearly independent and, thus, the only possibility is $kx = y$. Therefore, we find that $x$ and $y$ are not disjoint, that is, $x = y$. Moreover, since all the factorizations of $b_1$ are isolated (Proposition \ref{prop:all-i}), we note that $x \in \mathrm{Z}(b_1)$.
  
  The fact that $S$ is Cohen-Macaulay follows from Theorem \ref{thm:alpha-m+1}. \qedhere

\end{proof}

As a consequence of the previous lemma, we give a more precise version of Theorem \ref{thm:betti-sorted} for Betti divisible semigroups.


\begin{corollary} \label{cor:betti-divisible:simplicial:presentation}
  Let $S$ be a simplicial affine semigroup of $\mathbb{N}^r$ minimally generated by $\{n_1, \ldots, n_{r+m}\}$. The following statements are equivalent:
  \begin{enumerate}
  	\item \label{item:betti-divisible} $S$ is Betti divisible and $\min \mathrm{Betti}(S) \not \in \Ap(S; \{n_1, \ldots, n_r\})$;
    \item \label{item:betti-isolated-divisible} $S$ is Betti-isolated divisible and $\min \mathrm{Betti}(S) \not \in \Ap(S; \{n_1, \ldots, n_r\})$;
  	\item \label{item:betti-divisible:presen} $S$ admits a minimal presentation of the form
    \[ \{(c_{r+1}, x)\} \cup \{(c_i \mathbf{e}_i, q_i c_{i-1} \mathbf{e}_{i-1}) \colon i \in \{r+2, \ldots, r+m\}\}\]
    for a rearrangement of the $m$ latest minimal generators $n_{r+1}, \ldots, n_{r+m}$,  some non negative integers  $q_{r+2}, \ldots, q_{r+m}$ and $x \in \langle \mathbf{e}_1, \ldots, \mathbf{e}_r \rangle$;
  	\item \label{item:betti-divisible:presen:2} $S$ admits a minimal presentation of the form
  	\[ \{(a_{r+1}, x)\} \cup \{(a_i \mathbf{e}_i, q_i a_{i-1} \mathbf{e}_{i-1}) \colon i \in \{r+2, \ldots, r+m\}\}\]
  	for a rearrangement of the $m$ latest minimal generators $n_{r+1}, \ldots, n_{r+m}$, some non negative integers $a_{r+1}, \ldots, a_{r+m}, q_{r+2}, \ldots, q_{r+m}$ and $x \in \langle \mathbf{e}_1, \ldots, \mathbf{e}_r \rangle$.
  \end{enumerate}
\end{corollary}
\begin{proof}
  It is clear that \ref{item:betti-divisible} implies
  \ref{item:betti-isolated-divisible} and \ref{item:betti-divisible:presen} implies \ref{item:betti-divisible:presen:2} hold true.  
  
 \begin{enumerate}[leftmargin=10pt]
 \item[\ref{item:isolated-betti-sorted}\kern-3pt] implies \ref{item:betti-sorted:presen}. Since $S$ is Betti-isolated sorted and Cohen-Macaulay (Lemma \ref{lem:betti-divisible:cm}), Theorem \ref{thm:betti-sorted} can be applied, obtaining a rearrangement of the $m$ lastest minimal generators $n_{r+1}, \ldots, n_{r+m}$ such that $S$ admits a minimal presentation of the form
    \[ \left\{(c_{r+1} \mathbf{e}_{r+1}, \sum\nolimits_{j = 1}^{r}a_j^{r+1} \mathbf{e}_j \right\} \cup \left\{(c_i \mathbf{e}_i, \sum\nolimits_{j = 1}^{i-1} a_j^i \mathbf{e}_j + c_{i-1} \mathbf{e}_{i-1}) \colon  i \in \{r+2, \ldots, r+m\}\right\}.\]
    Since $S$ is Betti-isolated divisible, for each $i \in \{r+2, \ldots, r+m\}$, the fact that $c_{i-1} n_{i-1} \le_S c_i n_i$ is equivalent to  $c_{i-1} n_{i-1} \mid c_i n_i$. Therefore, there is a factorization $q_i c_{i-1} \mathbf{e}_{i-1} \in \mathrm{Z}(c_i n_i)$, which is in the same R-class as $\sum\nolimits_{j = 1}^{i-1} a_j^i \mathbf{e}_j + c_{i-1} \mathbf{e}_{i-1}$. As a consequence, the set
    \[ \{(c_{r+1}, x)\} \cup \{(c_i \mathbf{e}_i, q_i c_{i-1} \mathbf{e}_{i-1}) \colon i \in \{r+2, \ldots, r+m\}\}\]
    is a minimal presentation of $S$, where $x = \sum\nolimits_{j = 1}^{r}a_j^{r+1} \mathbf{e}_j$.
  \item[\ref{item:betti-divisible:presen:2}\kern-3pt] implies \ref{item:betti-divisible}. The Betti elements of $S$ are $\{a_i n_i \colon i \in \{r+1, \ldots, r+m\}\}$ and $q_i a_{i-1} n_{i-1} = a_i n_i$ for every $i \in \{r+2, \ldots, r+m\}$. That is, $S$ is Betti divisible. \qedhere
  \end{enumerate}
\end{proof}

Regarding numerical semigroups, we obtain the following particular case of Corollary \ref{cor:betti-sorted}.

\begin{corollary} \label{cor:betti-divisible:presen:ns}
  Let $S$ be a numerical semigroup. Let $(n_1, \ldots, n_e)$ be an arrangement of the minimal generators of $S$ such that $c_1 n_1 \le c_2 n_2 \le \cdots \le c_e n_e$. The following statements are equivalent:
  \begin{enumerate}
  	\item \label{item:betti-divisible:b} $S$ is Betti divisible;
    \item \label{item:betti-divisible:beta} $S$ is Betti-isolated divisible;
  	\item \label{item:betti-divisible:presentation} $S$ admits a minimal presentation of the form
  	\[\{(c_i \mathbf{e}_i, q_i c_{i-1} \mathbf{e}_{i-1}) \colon i \in \{2, \ldots, e\}\},\]
  	where $q_2, \ldots, q_e \in \mathbb{N}$;
  	\item \label{item:betti-divisible:presentation:2} $S$ admits a minimal presentation of the form
  	\[\{(a_i \mathbf{e}_i, q_i a_{i-1} \mathbf{e}_{i-1}) \colon i \in \{2, \ldots, e\}\},\]
  	where $a_1, \ldots, a_e, q_2, \ldots, q_e \in \mathbb{N}$.
  \end{enumerate}
\end{corollary}

Some Betti divisible semigroups can also be characterized in terms of gluings.

\begin{corollary} \label{cor:betti-divisible:gluing}
 Let $S$ be a simplicial affine semigroup of $\mathbb{N}^r$ minimally generated by $\{n_1, \ldots, n_{r+m}\}$. The following statements are equivalent:
  \begin{enumerate}
      \item \label{item:betti-divisible:gluing:a} $S$ is Betti divisible and $\min \mathrm{Betti}(S) \not \in  \Ap(S; \{n_1, \ldots, n_r\})$;
      \item \label{item:betti-divisible:gluing:b} $S$ is the gluing of a Betti divisible simplicial affine semigroup $S'$ of $\mathbb{N}^r$ and $\langle d \rangle \subset \mathbb{N}^r$, where either the codimension of $S'$ is $0$ or $\min \mathrm{Betti}(S') \not \in  \Ap(S; \{n_1, \ldots, n_r\})$ and $\max \mathrm{Betti}(S') \mid c_d d$ for $c_d = \min \{k \in \mathbb{Z}^+ : k d \in S'\}$. 
  \end{enumerate}
\end{corollary}
\begin{proof} Note that the semigroups considered are Cohen-Macaulay (Lemma \ref{lem:betti-divisible:cm}).
\begin{enumerate}[leftmargin=10pt]
   \item[\ref{item:betti-divisible:gluing:a}\kern-3pt] implies \ref{item:betti-divisible:gluing:b}. According to Corollary \ref{cor:betti-sorted:gluing}, $S$ is the gluing of a Betti sorted simplical affine semigroup $S'$ of $\mathbb{N}^r$ and $\langle d \rangle \subset \mathbb{N}^r$. If $S' = \mathbb{N}^r$, then we are done. Otherwise, Corollary \ref{cor:betti-sorted:gluing} states that  $\min \mathrm{Betti}(S') \not \in  \Ap(S; \{n_1, \ldots, n_r\})$ and $\max \mathrm{Betti}(S') \le_S d$. In light of the equality \eqref{eq:betti}, the Betti elements of $S$ are $\mathrm{Betti}(S') \cup \{d\}$ and, thus, $S'$ is Betti divisible and $\max \mathrm{Betti}(S') \mid d$.
   \item[\ref{item:betti-sorted:gluing:b}\kern-3pt] implies \ref{item:betti-sorted:gluing:a}. In view of Corollary \ref{cor:betti-sorted:gluing}, we have $\min \mathrm{Betti}(S) \not \in  \Ap(S; \{n_1, \ldots, n_r\})$. Again, the Betti elements of $S$ are $\mathrm{Betti}(S') \cup \{d\}$, which are totally ordered by divisibility.
  \qedhere
\end{enumerate}
\end{proof}

\begin{corollary} \label{cor:betti-divisible:gluing:ns}
  A numerical semigroup $S$ is Betti divisible if and only if there is there is another Betti divisible numerical semigroup $S'$ such that $S = d S' +_{d n} n \mathbb{N}$, where $\max \mathrm{Betti}(S') \mid n$.
\end{corollary}

Now we wonder whether there are Betti sorted semigroups that are not Betti divisible. One can construct any amount simplical affine semigroups of such a nature by applying Corollary \ref{cor:betti-sorted:gluing} and requiring that $d$ is not a multiple of $\max \mathrm{Betti}(S')$. 

In the rest of the section we focus on Betti divisible numerical semigroups. Corollary \ref{cor:betti-divisible:gluing:ns} provides us with a tool to construct Betti divisible numerical semigroups with an arbitrary number of Betti elements. The following lemma gives a wide family of Betti divisible numerical semigroups. Indeed, in Theorem \ref{thm:betti-divisible:generators} we prove that every Betti divisible numerical semigroups is of this form.

\begin{lemma} \label{lem:betti-divisible:family}
  Let $e$ and $k$ be positive integers such that $e \ge 2$ and $k \le e-1$. Let $a_1, \ldots, a_e, f_1, \ldots, f_e$ be  positive integers which verify
  \begin{enumerate}
  	\item $a_1, \ldots, a_e$ are pairwise relatively prime,
  	\item $1 = f_1 \mid f_2 \mid \cdots \mid f_e$,
  	\item $f_i$ and $a_i$ are relatively prime for every $i \in \{1, \ldots, e\}$.
  \end{enumerate}
  For each $i \in \{1, \ldots, e\}$ set $n_i = f_i \prod\nolimits_{j = 1}^{e} a_j / a_i$. The numerical semigroup $S = \langle n_1, \ldots, n_e \rangle$ is Betti divisible. Furthermore, it has embedding dimension $e$ and its Betti elements are
  \[\mathrm{Betti}(S) = \left\{f_i \prod\nolimits_{j = 1}^{e} a_j \colon i \in \{2, \ldots, e\}\right\}.\]
\end{lemma}
\begin{proof}
  Set $i \in \{1, \ldots, e\}$. Recall the definition of $d_i$, $\bar{c}_i$ and $c_i^*$ given in Section \ref{sec:pre:free}. Note that $n_j$ and $a_j$ are relatively prime and, thus, we have $d_i = \gcd(n_1, \ldots, n_{i-1}) = \prod_{j = i}^e a_j$. Since $d_i$ does not divides $n_i$, we find that $n_i \not \in \langle n_1, \ldots, n_{i-1}\rangle$. From the arbitrary choice of $i$ we conclude that the set $\{n_1, \ldots, n_e\}$ is the minimal system of generators of $S$, that is, $\mathrm{e}(S) = e$. Note that $n_1 \mid a_i n_i$, which can be used to easily compute $c_i^*$:
  \begin{align*}
    \bar{c}_i & = \min \{c \in \mathbb{Z}^+ \colon c n_i \text{ is a multiple of } d_i\} =  d_i / d_{i+1} = a_i, \\
    c_i^* & = \min \{c \in \mathbb{Z}^+ \colon c n_i \in \langle n_1, \ldots, n_{i-1} \rangle \} = a_i.
  \end{align*}
  Finally, by applying Lemma \ref{lem:free} we find that $S$ is free and its Betti elements are 
  \[\mathrm{Betti}(S) = \{a_2 n_2, \ldots, a_e n_e\} = \left\{f_i \prod\nolimits_{j = 1}^{e} a_j \colon i \in \{2, \ldots, k\}\right\}. \qedhere\]
\end{proof}

\begin{remark}
The integers $f_1, \ldots, f_e$ can be repeated. For example, if $f_i = 1$ for every $i \in \{1, \ldots, e\}$, then we obtain a numerical semigroup with only one Betti element. Indeed, the number of different Betti elements of $S$ equals the number of different integers in the sequence $f_2, \ldots, f_e$.
\end{remark}

\begin{example} \label{ex:betti-divisible:2}
	We recover the semigroup given in Example \ref{ex:betti-divisible:1}. Note that $b_1 = 2 \cdot 3 \cdot 5 \cdot 7$ and  $b_2 = 2 \cdot 2 \cdot 3 \cdot 5 \cdot 7$. Set $a_1 = 7$, $a_2 = 5$, $a_3 = 2$, $a_4 = 3$, $f_2 = 1$, $f_3 = 1$ and $f_4 = 2$. We have $30 = a_2 a_3 a_4$, $42 = a_1 a_3 a_4$, $105 = a_1 a_2 a_4$ and $140 = f_2 a_1 a_2 a_3$. Consequently, $S$ is in the family given in Lemma \ref{lem:betti-divisible:family}. 
\end{example}

\begin{theorem} \label{thm:betti-divisible:generators}
  Let $S$ be a numerical semigroup. Let $(n_1, \ldots, n_e)$ be an arrangement of the minimal generators of $S$ such that $c_1 n_1 \le c_2 n_2 \le \cdots \le c_e n_e$. Then, $S$ is Betti divisible if and only if there are positive integers $a_1, \ldots, a_e, f_1, \ldots, f_e$ which verify
  \begin{enumerate}
  	\item $a_1, \ldots, a_e$ are pairwise relatively prime,
  	\item $1 = f_1 = f_2 \mid \cdots \mid f_e$,
  	\item $f_i$ and $a_i$ are relatively prime for every $i \in \{1, \ldots, e\}$,
    \item $n_i = f_i \prod_{j = 1}^e a_j / a_i$ for every $i \in \{1, \ldots, e\}$.
  \end{enumerate}
  If it is the case, then $a_i = c_i$ for every $i \in \{1, \ldots, e\}$ and  
  \[\mathrm{Betti}(S) = \left\{f_i \prod\nolimits_{j = 1}^{e} c_j \colon i \in \{2, \ldots, e\}\right\}.\]
\end{theorem}
\begin{proof}
  Let us assume that $S$ is Betti divisible. Then, in view of Corollary \ref{cor:betti-sorted}, the set $\mathrm{Betti}(S)$ equals $\{c_1 n_2 = c_2 n_2 \mid \cdots \mid c_e n_e\}$. Moreover, since $S$ is $\alpha$-rectangular for $n_1$, we find that $n_1 = \prod_{j = 2}^e c_j$ and $c_i = c_i^*$ for every $i \in \{2, \ldots, e\}$ (Theorems \ref{thm:alpha:c} and \ref{thm:alpha:free}). For each $i \in \{1, \ldots, e\}$ define $f_i = c_i n_i / (c_1 n_1)$. Note that $f_1 = f_2 = 1$ and  $f_3 \mid \cdots \mid f_e$. Set $i \in \{1, \ldots, e\}$. We have $n_i = f_i c_1 n_1 / c_i =  f_i \prod_{j = 1; j \ne i}^e c_j$. Note that $c_i$ divides $\gcd(n_1, \ldots, n_{i-1}, n_{i+1}, \ldots, n_e)$. Since $\gcd(n_1, \ldots, n_e) = 1$, we obtain $\gcd(c_i, n_i) = 1$ and, in particular,  $c_i$ and $f_i$ are relatively prime. Moreover, from the arbitrary choice of $i$ we derive that $c_1, \ldots, c_e$ are pairwise relatively prime.
  
  Now let us consider a numerical semigroup $S$ which satisfies the four conditions. In Lemma \ref{lem:betti-divisible:family} we proved that this semigroup is Betti divisible and computed its Betti elements. Recall that we obtained $c_i^* = \bar{c}_i = a_i$ for every $i \in \{2, \ldots, e\}$. In light of Theorem \ref{thm:alpha:c}, we find that $a_i = c_i$ for every $i \in \{2, \ldots, e\}$. Moreover, we have $a_1 n_1 = c_2 n_2$ and, thus, we obtain $c_1 = a_1$.
\end{proof}

As a consequence, we recover the following characterization of those numerical semigroups with only one Betti element, which were studied in \cite{single-betti}.

\begin{corollary}[{\cite[Theorem 12]{single-betti}}] \label{cor:single-betti}
  Let $S$ be a numerical semigroup minimally generated by the set of positive integers $\{n_1, n_2, \ldots, n_e\}$. The following conditions are equivalent:
  \begin{enumerate}
  \item the semigroup $S$ only has one Betti element;
  \item there exist $a_1, a_2, \ldots, a_e$ pairwise relatively prime positive integers such that $n_i = \prod_{j \ne i} a_j$.
  \end{enumerate}  
  If it is the case, then $a_i = c_i$ for every $i \in \{1, \ldots, e\}$ and the Betti element of $S$ is $\prod_{j = 1}^e c_j$.
\end{corollary}

The following result gives another interesting property of Betti divisible numerical semigroups.

\begin{theorem}\label{thm:betti-divisible:free}
Let $S$ be a numerical semigroup. Consider the following assertions:
\begin{enumerate}
	\item \label{item:betti-divisible:s} $S$ is Betti divisible;
	\item \label{item:betti-divisible:gluing} for any non trivial partition $\{A,B\}$ of the minimal system of generators of $S$, $S$ is the gluing of the  Betti divisible numerical semigroups $\langle A / \gcd(A) \rangle$ and $\langle B / \gcd(B) \rangle$;
	\item \label{item:betti-divisible:free} $S$ is free for any arrangement of the minimal generators.
\end{enumerate}
Then \ref{item:betti-divisible:s} implies \ref{item:betti-divisible:gluing}, and \ref{item:betti-divisible:gluing} implies \ref{item:betti-divisible:free}. 
\end{theorem}
\begin{proof}
  Let $S$ be a numerical semigroup minimally generated by $\{n_1,\dots,n_e\}$. 

   \begin{enumerate}[leftmargin=10pt]
   \item[\ref{item:betti-divisible:s}\kern-3pt] implies \ref{item:betti-divisible:gluing}. Let $\{n_1, \ldots, n_e\}$ be the minimal generators of $S$ and let $\{A,B\}$ be a non trivial partition of this set. Let $J_A = \{j : n_j \in A\}$ and $J_B = \{j : n_j \in B\}$. In light of Theorem \ref{thm:betti-divisible:generators}, we find that $\gcd(A) = f_{\min J_A}\prod_{j \in J_B} c_j$ and $\gcd(B) = f_{\min J_B} \prod_{j \in J_A} c_j$, where $f_1, \ldots, f_e$ are given in Theorem \ref{thm:betti-divisible:generators}. Since $S$ is a numerical semigroup, the integers $\gcd(A)$ and $\gcd(B)$ are relatively prime. Moreover, we note that $\gcd(A) \gcd(B) = f_{\min J_A} f_{\min J_B} \prod_{j = 1}^e c_j$ is in $\langle A \rangle \cap \langle B \rangle$. Therefore, $S$ is the gluing of the numerical semigroups $\langle A / \gcd(A) \rangle$ and $\langle B / \gcd(B) \rangle$. Finally, from equation \eqref{eq:betti} it follows
   \[ \mathrm{Betti}(S) = \gcd(A) \mathrm{Betti}(\langle A / \gcd(A) \rangle) \cup \gcd(B) \mathrm{Betti}(\langle B / \gcd(B) \rangle). \]
   Thus, the semigroups $\langle A / \gcd(A) \rangle$ and $\langle B / \gcd(B) \rangle$ are Betti divisible.
   
   
   \item[\ref{item:betti-divisible:gluing}\kern-3pt] implies \ref{item:betti-divisible:free}. The proof is carried out by induction on $\mathrm{e}(S)$. If $S = \mathrm{N}$, then the assertion is trivial. Let us assume that the result holds for numerical semigroups with embedding dimension smaller or equal than $\mathrm{e}(S) \ge 2$. Let $(n_1, \ldots, n_e)$ be a arrangement of the minimal generators of $S$ such that $c_1 n_1 \le \cdots \le c_e n_e$. Then $S$ is the gluing of $S_e = \langle n_1 / c_e, \ldots, n_{e-1} / c_e\rangle$ and $\mathbb{N}$, where $S_e$ is Betti divisible. By the induction hypothesis, $S_e$ is free and, thus, $S$ is free for $(n_1, \ldots, n_e)$ (Lemma \ref{lem:free}). \qedhere

\end{enumerate}
\end{proof}

As noticed by I. García-Marco and C. Tatakis in \cite{robust}, the  implication c) implies a) in Theorem~\ref{thm:betti-divisible:free} does not hold. There are semigroups free for any arrangement of the minimal generators that are not Betti divisible.
\begin{verbatim}	
gap> s:=NumericalSemigroup(390,546,770,1155);;
gap> IsUniversallyFree(s);
true
gap> BettiElements(s);
[ 2310, 2730, 30030 ]
\end{verbatim}

Finally, we give a novel characterization of those numerical semigroups with a unique Betti elements that involves all the concepts developed in this work. Recall that another characterization of these semigroups was given in Corollary \ref{cor:iso:one-betti}.

\begin{theorem}\label{thm:single-betti-alpha}
  Let $S$ be a numerical semigroup. The following statements are equivalent:
  \begin{enumerate}
      \item \label{item:single-betti} $S$ only has one Betti element;
      \item \label{item:single-betti:c} $S$ is $c$-rectangular for any minimal generator and $\mathrm{i}_b(S) = \mathrm{e}(S)$;
      \item \label{item:single-betti:alpha} $S$ is $\alpha$-rectangular for any minimal generator.
  \end{enumerate}
\end{theorem}
\begin{proof} ~
   \begin{enumerate}[leftmargin=10pt]
   \item[\ref{item:single-betti}\kern-3pt] implies \ref{item:betti-sorted:c}. Let $(n_1, \ldots, n_e)$ be a rearrangement of the minimal generators of $S$. In light of Corollary \ref{cor:betti-divisible:presen:ns}, $S$ admits a minimal presentation of the form 
   \[\{(c_i \mathbf{e}_i, c_{i-1} \mathbf{e}_{i-1}) : i \in \{2, \ldots, e\}\}.\]
   Therefore, we have $\mathrm{I}_b(S) = \{c_1 \mathbf{e}_1, \ldots, c_e \mathbf{e}_e\}$. Moreover, we have $c_i = c_i^*$ for every $i \in \{2, \ldots, e\}$. In light of Lemma \ref{lem:free}, $S$ is free for that arrangement and $S$ is $c$-rectangular for $n_1$. 

   \item[\ref{item:single-betti:c}\kern-3pt] implies \ref{item:betti-sorted:alpha}. It follows from Theorem \ref{thm:alpha:c}.

   \item[\ref{item:single-betti:alpha}\kern-3pt] implies \ref{item:single-betti}. Let $\{n_1, \ldots, n_e\}$ be a minimal system of generators of $S$. Since $S$ is $\alpha$-rectangular for any minimal generator $n_i$, we obtain $n_i = \prod_{i \ne j} c_j$ (Theorem \ref{thm:alpha:c}). Therefore, the characterization of numerical semigroups with only one Betti element given in Corollary \ref{cor:single-betti} is satisfied. \qedhere
   \end{enumerate}
\end{proof}

\section{Further research}





The concept of $\beta$-rectangular and $\gamma$-rectangular numerical semigroup introduced in \cite{ci:classes:1} can also be generalized to the context of simplicial affine semigroups as we did with $\alpha$-rectangular semigroups. Nevertheless, we did not include these generalizations in this paper since they are out of the scope of isolated factorizations. The study of these simplicial affine semigroups may be an interesting topic for further research. Thirdly, it is possible that some non-unique factorization invariants (\cite{geroldinger-hk, overview-nuf} and \cite[Chapter 5]{ns-app})  may be computed or bounded for Betti divisible  semigroups. Again this topic is not related with isolated factorizations and, therefore, we did not dig into this topic in the present work. Finally, we introduced rectangular semigroups and, in particular, $c$-rectangular semigroups which may be also be a focus of further investigations.

\printbibliography

\end{document}